\documentclass[11pt]{amsart}
\usepackage[dvipsnames,svgnames,x11names]{xcolor}
\usepackage[T1]{fontenc}
\usepackage{textgreek,bbm,url,graphicx,verbatim,amssymb,enumerate,mathtools}
\usepackage{amsmath,amsthm,amssymb}
\usepackage{amsfonts}
\usepackage[pagebackref,colorlinks,citecolor=Mahogany,linkcolor=Mahogany,urlcolor=Mahogany,filecolor=Mahogany]{hyperref}
\usepackage[capitalize]{cleveref}
\usepackage[mathscr]{euscript}

\usepackage{fancyhdr}
\usepackage[a4paper,
bindingoffset=0.2in,
left=0.8in,
right=1in,
top=1in,
bottom=1in,
footskip=.25in,
headheight=15pt]{geometry}
\usepackage{enumitem}
\setlist[enumerate,1]{label=\textup{(\arabic*)}}

\makeatletter
\DeclareFontFamily{OMX}{MnSymbolE}{}
\DeclareSymbolFont{MnLargeSymbols}{OMX}{MnSymbolE}{m}{n}
\SetSymbolFont{MnLargeSymbols}{bold}{OMX}{MnSymbolE}{b}{n}
\DeclareFontShape{OMX}{MnSymbolE}{m}{n}{
<-6>  MnSymbolE5
<6-7>  MnSymbolE6
<7-8>  MnSymbolE7
<8-9>  MnSymbolE8
<9-10> MnSymbolE9
<10-12> MnSymbolE10
<12->   MnSymbolE12
}{}
\DeclareFontShape{OMX}{MnSymbolE}{b}{n}{
<-6>  MnSymbolE-Bold5
<6-7>  MnSymbolE-Bold6
<7-8>  MnSymbolE-Bold7
<8-9>  MnSymbolE-Bold8
<9-10> MnSymbolE-Bold9
<10-12> MnSymbolE-Bold10
<12->   MnSymbolE-Bold12
}{}

\let\llangle\@undefined
\let\rrangle\@undefined
\DeclareMathDelimiter{\llangle}{\mathopen}%
{MnLargeSymbols}{'164}{MnLargeSymbols}{'164}
\DeclareMathDelimiter{\rrangle}{\mathclose}%
{MnLargeSymbols}{'171}{MnLargeSymbols}{'171}
\makeatother

\def\XXint#1#2#3{{\setbox0=\hbox{$#1{#2#3}{\int}$ }
\vcenter{\hbox{$#2#3$ }}\kern-.6\wd0}}

\newtheorem{theorem}{Theorem}[section]
\newtheorem{lemma}[theorem]{Lemma}
\newtheorem{proposition}[theorem]{Proposition}
\newtheorem{corollary}[theorem]{Corollary}
\newtheorem*{corollary*}{Corollary}

\newtheorem{atheorem}{Theorem}

\newtheorem{acorollary}[atheorem]{Corollary}

\theoremstyle{definition}
\newtheorem{definition}[theorem]{Definition}

\theoremstyle{remark}

\newtheorem{remark}[theorem]{Remark}

\crefname{lemma}{Lemma}{Lemmas}
\crefname{proposition}{Proposition}{Propositions}
\crefname{corollary}{Corollary}{Corollaries}
\crefname{atheorem}{Theorem}{Theorems}
\crefname{acorollary}{Corollary}{Corollaries}

\newcommand{\subref}[2]{\hyperref[#2]{\ref{#1}.\ref{#2}}}

\numberwithin{equation}{section}

\newcommand{\R}{\mathbb{R}}

\newcommand{\N}{\mathbb{N}}

\newcommand{\ind}{\mathbf{1}}

\newcommand{\RR}{\mathcal{R}}

\definecolor{darkgreen}{HTML}{005239}

\author{Francesco D'Emilio}
\author{Brett D. Wick}
\address{Francesco D'Emilio\hfill\break\indent
Department of Mathematics \hfill\break\indent
Washington University in St. Louis \hfill\break\indent
One Brookings Drive \hfill\break\indent
St. Louis, MO 63130 USA}
\email{demilio@wustl.edu}
\address{Brett D. Wick\hfill\break\indent
Department of Mathematics \hfill\break\indent
Washington University in St. Louis \hfill\break\indent
One Brookings Drive \hfill\break\indent
St. Louis, MO 63130 USA}
\email{bwick@wustl.edu}

\title{Martingale Transforms and Compensated Bellman Estimates for Dunkl Riesz Transforms}

\thanks{\textbf{Funding Acknowledgment:} Francesco D'Emilio is partially supported by the Simons Foundation through a Simons Dissertation Fellowship. Brett D. Wick's research is supported in part by National Science Foundation DMS award \#2349868 and Australian Research Council grant DP 220100285.}

\begin{document}
\begin{abstract}
We develop a martingale-transform framework for Dunkl harmonic analysis and apply it to prove $L^p$ estimates for the Dunkl Riesz transforms using the martingale decomposition of the Dunkl process. A fundamental difference from the classical Brownian setting is that the martingale representing the Dunkl Riesz transform of a function is not, in general, differentially subordinated to the Poisson martingale of the function itself. Our main argument bypasses this obstruction by applying Burkholder's Bellman function directly. The possible positive defect of the continuous It\^o drift is compensated by the negative contribution produced by the reflection jumps. This yields $L^p$ estimates for single Dunkl Riesz transforms for $1<p<\infty$, and vector-valued estimates for $p\geq2$ with a spectral dependence on the root system. For $G$-invariant functions, differential subordination can be recovered and the resulting estimates are independent of the root system.
\end{abstract}
\subjclass[2020]{Primary 42B20; Secondary 42B15, 60G44, 60G46}

\keywords{Dunkl Riesz transforms, Dunkl processes,
martingale transforms, Bellman functions,
differential subordination, dimension-free estimates}
\maketitle
\section{Introduction}

The purpose of this paper is to generalize the martingale method of Gundy and Varopoulos \cite{GV} to the Dunkl setting and to use it to obtain quantitative $L^p$ estimates for Dunkl Riesz transforms. In the Euclidean framework, the basic idea is to lift a function $f$ to the upper half-space through its Poisson extension and evaluate it along a space-time Brownian path. The resulting process is a martingale whose stochastic differential is determined by the gradient of the extension. Applying a suitable matrix to this stochastic gradient produces a martingale transform, and a conditional projection at the boundary recovers the singular integral under consideration. For the Riesz transforms, this construction is particularly well adapted to the geometry of Brownian motion: with a suitable choice of representation, the transformed martingale is differentially subordinate and orthogonal to the original Poisson martingale. Sharp martingale inequalities of Burkholder and Ba\~nuelos--Wang, see for example \cite{burk,bur88,BW95,BW96,W95}, can then be transferred to the corresponding singular integral operators. For more applications in probability theory and harmonic analysis, see \cite{Ban2,osekowski,dompet} and the references therein. \\

Dunkl harmonic analysis provides a natural nonlocal analogue of classical Fourier analysis in the presence of finite reflection symmetries. Given a root system $\RR\subset\R^n$ and an associated multiplicity function $\kappa$, the usual partial derivatives are replaced by the differential-difference operators introduced by Dunkl in \cite{Dunkl87}. These operators give rise to analogues of the Laplacian, Fourier transform, heat and Poisson semigroups, and Riesz transforms. The theory also appears naturally in mathematical physics, particularly in connection with Calogero--Moser--Sutherland type systems; see for example \cite{LV}.

From a probabilistic point of view, the Dunkl Laplacian generates a Markov process which combines a Brownian part with a purely discontinuous component produced by reflections across the hyperplanes of the root system. A precise description of this process was obtained by Gallardo and Yor in \cite{GallardoYor2006_Chaotic}. Their martingale decomposition expresses the Dunkl process as the sum of a standard Brownian motion and a family of orthogonal purely discontinuous martingales corresponding to different reflection directions. Unlike in the Euclidean setting, the stochastic differential of the Dunkl Poisson martingale is naturally described not by the Dunkl gradient itself, but by the gradient associated with the carr\'e du champ operator, which contains coordinates that measure the difference of the function across each reflection hyperplane. Hence, the stochastic space on which a Dunkl martingale transform acts is larger than the physical space, due to the additional jump geometry.

One of the structural points of the paper is to make this correspondence explicit. Starting from a matrix representation of a Euclidean Fourier multiplier, we introduce a canonical lift to the enlarged stochastic space associated with the Dunkl process. The conditional projection of the corresponding martingale transform is then a Dunkl Fourier multiplier with the same symbol. In particular, this construction identifies the matrices associated with the first-order Dunkl Riesz transforms and also explains the different stochastic representations of the same multiplier which will be important below.

\subsection{Main results and technical novelties}

For $j=1,\ldots,n$, the Dunkl Riesz transform $R_j$ is defined by
\[
\mathcal F_\kappa(R_jf)(\xi)
=
-i\frac{\xi_j}{|\xi|}
\mathcal F_\kappa f(\xi),
\qquad \xi\neq0,
\]
and we write the associated Riesz vector as
\[
\mathbf Rf=(R_1f,\ldots,R_nf).
\]

The study of Riesz transforms in the Dunkl setting goes back to \cite{thangxu} and \cite{amsi}, yet the first dimension-free estimates for the Dunkl Riesz vector for a general root system were obtained by Hejna \cite{Dimensionless} through a bilinear Bellman function argument. Her main result gives, for $1<p<\infty$,
\[
\|\mathbf Rf\|_{L^p(dw)}
\leq
144(p^*-1)
\left(
\sum_{\alpha\in\RR}\kappa(\alpha)+2^7
\right)
\|f\|_{L^p(dw)},
\]
where $p^*=\max\{p,p'\}$. If $f$ is $G$-invariant, there is no dependence on the root system:
\[
\|\mathbf Rf\|_{L^p(dw)}
\leq
144(p^*-1)\|f\|_{L^p(dw)}.
\]

The martingale representation developed here leads to estimates with a different dependence on the geometry of the root system. For a fixed coordinate $j$, set
\[
\theta_j
:=
\left(
1+
\sum_{\alpha\in\RR_+}
\kappa(\alpha)\alpha_j^2
\right)^{1/2}.
\]
In the following, we assume $|\alpha|=2$ for every $\alpha \in \RR_+$. The second main result is as follows.
\begin{atheorem}\label[atheorem]{single riesz}
For every $1<p<\infty$ and every $j\in\{1,\ldots,n\}$,
\[
\|R_jf\|_{L^p(dw)}
\leq
2(p^*-1)\theta_j
\|f\|_{L^p(dw)}.
\]
If $f\in L^p(dw)$ is $G$-invariant, then
\[
\|R_jf\|_{L^p(dw)}
\leq
2(p^*-1)\|f\|_{L^p(dw)}.
\]
\end{atheorem}

For the vector estimate, the natural root-system parameter is
\[
\Theta_\kappa
:=
\left\|
I_{\R^n}+M_\kappa
\right\|_{\mathrm{op}}^{1/2},
\qquad
M_\kappa:=\sum_{\alpha\in\RR_+}
\kappa(\alpha)\,\alpha\otimes\alpha.
\]

\begin{atheorem}\label[atheorem]{vector}
For every $p\geq2$,
\[
\|\mathbf Rf\|_{L^p(dw;\ell_n^2)}
\leq
2(p-1)\Theta_\kappa
\|f\|_{L^p(dw)}.
\]
Moreover, for every $1<p<\infty$ and every $G$-invariant
$f\in L^p(dw)$,
\[
\|\mathbf Rf\|_{L^p(dw;\ell_n^2)}
\leq
2(p^*-1)\|f\|_{L^p(dw)}.
\]
\end{atheorem}

Therefore, our result improves both the absolute constant and the dependence on the root geometry when $p\geq2$. In particular, we have
\[
\theta_j,\Theta_\kappa\leq\sqrt{1+2\gamma},
\qquad
\gamma:=\sum_{\alpha\in\RR_+}\kappa(\alpha).
\]
If the root system is irreducible and spans $\R^n$, then $M_\kappa=\frac{2\gamma}{n}I$ since it commutes with the reflection group, hence
\[
\Theta_\kappa=\sqrt{1+\frac{2}{n}\gamma}.
\]

In dimension one, the invariant estimate also gives a multiplicity-independent bound for arbitrary inputs. Indeed, the Dunkl Hilbert transform interchanges even and odd functions, so the estimate on even inputs extends to odd inputs by duality. Decomposing an arbitrary function into its even and odd parts gives the following consequence, proved in \cref{cor:rank-one-bound}.

\begin{acorollary}[The one-dimensional case]\label[acorollary]{cor:rank-one-intro}
Let $n=1$ and write $H_\kappa:=R_1$ for the Dunkl Hilbert transform. For every $1<p<\infty$ and every $f\in L^p(dw)$,
\[
\|H_\kappa f\|_{L^p(dw)}
\leq4(p^*-1)\|f\|_{L^p(dw)},
\]
uniformly in the multiplicity $\kappa\geq0$.
\end{acorollary}

Plancherel's theorem gives the exact normalization
$\|\mathbf Rf\|_2=\|f\|_2$. Interpolation with this identity yields a small improvement of the estimates near $p=2$; see \cref{cor:interpolated-riesz-bounds}.

The main difficulty in these results is that the classical martingale results of \cite{W95,BW96} are no longer directly applicable. In the Brownian setting, the matrix representing a Riesz transform is a contraction on the infinitesimal martingale increments, yielding differential subordination. In the Dunkl setting, the matrix representing $R_j$ is still skew-symmetric, yet it acts on the discontinuous reflection coordinates and produces jumps in the transformed martingale that are not pointwise comparable to the jumps of the source martingale, breaking subordination and preventing the application of the classical inequalities. However, its lower-triangular part $L_j$, which represents $\frac12R_j$, generates a continuous martingale. This representation loses the orthogonality associated with the full skew-symmetric matrix, but it is much better adapted to Burkholder's Bellman function even if differential subordination still fails.
Indeed, full differential subordination turns out to be stronger than what is needed to close the Bellman estimate. Let $p\geq2$ and denote by
\[
\mathcal D^{\mathcal B}:=\mathcal C^{\mathcal B}+\mathcal J^{\mathcal B}
\]
the predictable drift obtained from It\^o's formula for Burkholder's function $\mathcal B$, where $\mathcal J^{\mathcal B}$ denotes the compensated jump contribution. Differential subordination makes both terms nonpositive. In our case, $\mathcal C^{\mathcal B}$ need not be nonpositive, but its positive defect is controlled by the reflection energy. A finite-difference inequality for Burkholder's function, proved in \cref{subsec:burkholder-scalar-lemmas}, gives a negative jump contribution with the same coefficient. After compensation, this yields $\mathcal D^{\mathcal B}\leq0$. \\

The situation simplifies substantially for $G$-invariant functions. In this case, the Poisson martingale itself is continuous. The full skew-symmetric representation recovers orthogonality but not differential subordination, since its martingale may still contain reflection jumps, so this argument does not recover the classical Euclidean bound. Passing again to the lower-triangular representation guarantees differential subordination, so the usual Burkholder inequality applies without any root-system loss. For $p\geq2$, the same mechanism applies simultaneously to all the lower-triangular representations $L_j$, leading to the vector estimate in \cref{vector}. The general vector estimate below $2$ is not obtained by this argument. \\

\subsection{Structure of the paper}
The rest of the paper is organized as follows. We first recall the basic analytic ingredients of the Dunkl setting. In \cref{prob set up} we introduce the Dunkl process and its decomposition into continuous and reflection martingales, while in \cref{sect 4,sec:invariant-sharp-bounds} we introduce Dunkl martingale transforms and the matrix lifting procedure for symbols, and we study the special geometry of $G$-invariant functions. The main Bellman argument for general functions, including the cancellation between the continuous defect and the compensated reflection drift, is carried out in the final section.

\section{The Dunkl setting}

We begin with some standard definitions and operators in the Dunkl setting.
All the facts in this section can be found in
\cite{dunkl,dunklbook,Dunkl87,DziubanskiHejna2019,DziubanskiHejna2022,Rosler2002}. Let $\R^n$ be the Euclidean space equipped with the standard inner product
$\langle\cdot,\cdot\rangle$. For a non-zero vector $\alpha\in\R^n$, let
$\sigma_\alpha$ denote the reflection with respect to the hyperplane
$H_\alpha$ orthogonal to $\alpha$. A finite set
$\RR\subset\R^n\setminus\{0\}$ is called a reduced root system if
\[
\RR\cap\mathbb R\alpha=\{\pm\alpha\}
\qquad\text{and}\qquad
\sigma_\alpha(\RR)=\RR
\]
for every $\alpha\in\RR$. The finite group $G$ generated by the reflections
$\{\sigma_\alpha:\alpha\in\RR\}$ is called the reflection group associated
with $\RR$. A function $f$ is called $G$-invariant if
$f(gx)=f(x)$ for every $g\in G$ and $x\in\R^n$. Let us fix a positive subsystem $\RR_+\subset\RR$, so that
$\RR=\RR_+\cup(-\RR_+)$, and assume that the roots are normalized by
$|\alpha|^2=2$. Thus,
\[
\sigma_\alpha x
=
x-\langle\alpha,x\rangle\alpha,
\qquad x\in\R^n.
\]
Let $\kappa:\RR\to\R_{\geq0}$ be a multiplicity function, that is, a
$G$-invariant function on $\RR$. Recall
\[
\gamma:=\sum_{\alpha\in\RR_+}\kappa(\alpha)
\]
and define
\[
w_\kappa(x)
:=
\prod_{\alpha\in\RR_+}
|\langle\alpha,x\rangle|^{2\kappa(\alpha)},
\qquad
dw(x):=w_\kappa(x)\,dx.
\]

The function $w_\kappa$ is $G$-invariant and homogeneous of degree
$2\gamma$ and the measure $dw(x)=w_\kappa(x)dx$ is doubling.
Introduced by C.~Dunkl in \cite{Dunkl87}, the Dunkl operators associated
with $\RR$ and $\kappa$ are the differential-difference operators
\begin{equation}
T_\xi f(x)
=
\partial_\xi f(x)
+
\sum_{\alpha\in\RR_+}
\kappa(\alpha)\langle\alpha,\xi\rangle
\frac{f(x)-f(\sigma_\alpha x)}
{\langle\alpha,x\rangle},
\label{E:dunk-op}
\end{equation}
defined initially for $f\in C^1(\R^n)$ and $\xi\in\R^n$. On the
reflection hyperplanes, the quotient in \eqref{E:dunk-op} is understood
by continuous extension. We write $T_j:=T_{e_j}$ for the operators related to the standard basis
vectors $e_1,\ldots,e_n$. \\
In the following, we assume for simplicity that every $\alpha\in\RR_+$ satisfies $\kappa(\alpha)>0$ and write $m:=|\RR_+|$. All the results in the paper can be generalized to arbitrary root systems and multiplicity functions by restricting to the subset $\RR_+^\kappa\subset\RR_+$ of positive roots having positive weight. For $\alpha\in\RR_+$, define
\begin{equation}
J^\alpha f(x)
:=
\sqrt{\kappa(\alpha)}
\frac{f(x)-f(\sigma_\alpha x)}
{\langle\alpha,x\rangle},
\qquad
\beta_\alpha:=\sqrt{\kappa(\alpha)}\,\alpha.
\end{equation}
For a function $F$ defined on $\R^n\times[0,\infty)$ and sufficiently regular, we will again write
\[
J^\alpha F(x,s)
:=
\sqrt{\kappa(\alpha)}
\frac{F(x,s)-F(\sigma_\alpha x,s)}
{\langle\alpha,x\rangle}.
\]
The Dunkl gradient is then defined as
\[
\nabla_\kappa f
:=
(T_1f,\ldots,T_nf)^\top
=
\nabla f
+
\sum_{\alpha\in\RR_+}
J^\alpha f\,\beta_\alpha,
\]
where $\nabla$ is the standard Euclidean gradient.
The Dunkl operators commute pairwise and are skew-symmetric with respect
to the measure $dw$.
The Dunkl Laplacian is defined as
\[
\Delta_\kappa:=\sum_{j=1}^nT_j^2.
\]
For $f\in C_c^2(\R^n)$, one has the following expression
\begin{equation}
\Delta_\kappa f(x)
=
\Delta f(x)
+
2\sum_{\alpha\in\RR_+}
\kappa(\alpha)
\left(
\frac{\partial_\alpha f(x)}
{\langle\alpha,x\rangle}
-
\frac{f(x)-f(\sigma_\alpha x)}
{\langle\alpha,x\rangle^2}
\right),
\label{eq:dunkl-laplacian}
\end{equation}
where $\Delta$ is the Euclidean Laplacian and
$\partial_\alpha$ is the $\alpha$-directional derivative.
Another fundamental operator is the carr\'e du champ operator. For real-valued $f,g\in C^2(\R^n)$, this is defined as
\[
\Gamma_\kappa(f,g)
:=
\frac12
\left(
\Delta_\kappa(fg)
-
f\Delta_\kappa g
-
g\Delta_\kappa f
\right).
\]
A direct computation gives
\begin{equation}
\Gamma_\kappa(f,g)
=
\langle\nabla f,\nabla g\rangle
+
\sum_{\alpha\in\RR_+}
J^\alpha f\,J^\alpha g.
\end{equation}
In particular, if $\RR_+=\{\alpha_1,\ldots,\alpha_m\}$,
\[
\Gamma_\kappa(f)
:=
\Gamma_\kappa(f,f)
=
|\nabla f|^2
+
\sum_{\alpha\in\RR_+}|J^\alpha f|^2
=:
|\nabla_\Gamma f|^2,
\]
where the gradient associated with the
carr\'e du champ is
\begin{equation*}
\nabla_\Gamma f
:=
\bigl(
\nabla f,
J^{\alpha_1}f,\ldots,J^{\alpha_m}f
\bigr)^\top.
\end{equation*}

\subsection{The Dunkl Fourier transform and Dunkl convolution}

The Dunkl kernel $E_\kappa(x,y)$, for $x,y\in\R^n$, plays the role of
the exponential function $e^{\langle x,y\rangle}$ in the Dunkl setting.
It is the unique analytic solution of
\[
T_\xi^xE_\kappa(x,y)
=
\langle\xi,y\rangle E_\kappa(x,y);
\qquad
E_\kappa(0,y)=1,
\]
where $T_\xi^xE_\kappa(x,y):=T_\xi(E_\kappa(\cdot,y))(x)$.
The kernel extends holomorphically to
$\mathbb C^n\times\mathbb C^n$ and is symmetric in its arguments. The Dunkl transform of $f\in L^1(dw)$ is then defined as
\begin{equation}
\mathcal F_\kappa f(\xi)
:=
c_\kappa^{-1}
\int_{\R^n}
f(x)E_\kappa(x,-i\xi)\,dw(x),
\qquad
\xi\in\R^n,
\label{eq:dunkl-transform}
\end{equation}
where
\[
c_\kappa
:=
\int_{\R^n}
e^{-|x|^2/2}w_\kappa(x)\,dx.
\]
The Dunkl transform extends to an isometric isomorphism on $L^2(dw)$,
and satisfies the Euclidean type identities
\[
\mathcal F_\kappa(T_jf)(\xi)
=
i\xi_j\mathcal F_\kappa f(\xi),
\qquad
\mathcal F_\kappa(\Delta_\kappa f)(\xi)
=
-|\xi|^2\mathcal F_\kappa f(\xi).
\]

For $x\in\R^n$, the Dunkl translation $\tau_x$ is defined on $L^2(dw)$
by
\[
\mathcal F_\kappa(\tau_xf)(\xi)
=
E_\kappa(i\xi,x)\mathcal F_\kappa f(\xi).
\]
For $f\in\mathcal S(\R^n)$, this gives
\[
\tau_xf(y)
=
c_\kappa^{-1}
\int_{\R^n}
E_\kappa(i\xi,x)
E_\kappa(i\xi,y)
\mathcal F_\kappa f(\xi)
w_\kappa(\xi)\,d\xi.
\]
Unlike the classical translation, $\tau_x$ is not positive in general
when $\kappa\neq0$. The Dunkl convolution of suitable functions $f,g$ is defined by
\[
(f*_\kappa g)(x)
:=
\int_{\R^n}
f(y)\tau_xg(-y)\,dw(y).
\]
With the normalization in \eqref{eq:dunkl-transform}, we have
\[
\mathcal F_\kappa(f*_\kappa g)
=
c_\kappa\,
\mathcal F_\kappa f\,
\mathcal F_\kappa g.
\]

\subsection{The Dunkl heat and Poisson semigroups}

The operator $\frac12\Delta_\kappa$ admits a self-adjoint realization on $L^2(dw)$ and generates a strongly continuous symmetric Feller semigroup
\[
H_tf
:=
e^{t\Delta_\kappa/2}f.
\]
On the Dunkl-transform side,
\[
\mathcal F_\kappa(H_tf)(\xi)
=
e^{-t|\xi|^2/2}\mathcal F_\kappa f(\xi).
\]
We can realize $H_tf$ also as the Dunkl convolution of $f$ with the kernel
\[
h_t(x)=c_\kappa^{-1}t^{-\mathbf N/2}e^{-|x|^2/(2t)},
\]
where $\mathbf N=n+2\gamma$ is the homogeneous dimension of the system. In other words
\[
H_tf(x)
=
\int_{\R^n}
h_t(x,y)f(y)\,dw(y),
\qquad
h_t(x,y):=\tau_xh_t(-y).
\]
For every $t>0$, the function $h_t(x,y)$ is smooth, positive,
symmetric in $x$ and $y$, and
\begin{equation}\label{E: dunkl heat conservative}
\int_{\R^n}h_t(x,y)\,dw(y)=1.
\end{equation}

The Poisson semigroup is defined by
\[
P_y
:=
e^{-y\sqrt{-\Delta_\kappa}},
\qquad y>0.
\]
For a suitable function $f$, its Poisson extension will be denoted by
\begin{equation}
U_f(x,y)
:=
P_yf(x)
=
\frac1{\sqrt\pi}
\int_0^\infty
e^{-t}
H_{y^2/(2t)}f(x)
\frac{dt}{\sqrt t}.
\label{eq:dunkl-poisson-subordination}
\end{equation}
If $f$ is bounded and continuous, then $U_f$ is the bounded solution of
\[
\begin{cases}
\partial_y^2U_f(x,y)+\Delta_\kappa U_f(x,y)=0,
& (x,y)\in\R^n\times(0,\infty),\\
U_f(x,0)=f(x),
& x\in\R^n.
\end{cases}
\]
The Dunkl-transform identity is
\begin{equation*}
\mathcal F_\kappa(U_f)(\xi,y)
=
e^{-y|\xi|}\mathcal F_\kappa f(\xi).
\end{equation*}

\subsection{Square functions}
We finally state some results about some square functions that will be used in the following arguments. Let $s$ be a positive integer such that $2s>\mathbf N+1$. Let $\varphi$ be a $C^{2s}(\R^n)$ function such that, for a certain $M>\lfloor\mathbf N\rfloor+1$, we have
\[
|\partial^\beta\varphi(x)|\leq C(1+|x|)^{-(M+\mathbf N)},
\qquad \beta\in\N_0^n,\quad |\beta|\leq2s,
\]
For $y>0$, let $\varphi_y(x)=y^{-\mathbf N}\varphi(x/y)$.
The carr\'e du champ square function is
\[
S_{\Gamma,\varphi}(f)(x)
=\bigg(\int_0^\infty y\,|\nabla_\Gamma(\varphi_y\ast_\kappa f)(x)|^2\,dy\bigg)^{\frac12}.
\]
Similarly, assume that
$\phi\in C^{2s+1}(\R^n)$ and
\[
|\partial^\beta\phi(x)|
\leq
C(1+|x|)^{-M-\mathbf N-1},
\qquad
|\beta|\leq2s+1,
\]
for some $M>\lfloor\mathbf N\rfloor+1$. The vertical square function is defined as
\[
S_{y,\phi}(f)(x)
=\bigg(\int_0^\infty y\,|\partial_y(\phi_y\ast_\kappa f)(x)|^2\,dy\bigg)^{\frac12}.
\]

The following result was proved in \cite{DzHe}.
\begin{theorem}\label{T: Dzhe}
For every $1<p<\infty$ and all $\phi,\varphi$ as above, there exists a positive constant $C_{p,\phi,\varphi}$ such that, for every $f\in L^p(dw)$,
\[
\|S_{y,\phi}f\|_p+\|S_{\Gamma,\varphi}f\|_p
\leq C_{p,\phi,\varphi}\|f\|_p.
\]
\end{theorem}

Now, for $z=(x,y)\in\R^n\times(0,\infty)$, define the augmented carr\'e du champ gradient as
\begin{equation*}
\widetilde\nabla F(z)
:=
\bigl(\nabla_\Gamma F(z),\partial_yF(z)\bigr)^\top.
\end{equation*}
The augmented square function for the Poisson extension is
\[
\widetilde G_\Gamma f(x)
=\bigg(\int_0^\infty y|\widetilde\nabla U_f(x,y)|^2\,dy\bigg)^{\frac12}.
\]
Notice that $\widetilde G_\Gamma f(x)\leq G_\Gamma f(x)+G_yf(x)$, where
\begin{align*}
G_\Gamma f(x)&=\bigg(\int_0^\infty y|\nabla_\Gamma U_f(x,y)|^2\,dy\bigg)^{\frac12},\\
G_yf(x)&=\bigg(\int_0^\infty y|\partial_yU_f(x,y)|^2\,dy\bigg)^{\frac12}.
\end{align*}

A direct consequence of the previous theorem is the following.
\begin{corollary}\label[corollary]{L:augmented-square-function-bound}
The square function $\widetilde G_\Gamma$ is $L^p(dw)$ bounded for any $1<p<\infty$.
\end{corollary}

\begin{proof}
It suffices to prove that both $G_\Gamma f$ and $G_yf$ are $L^p$ bounded.
Choose
\[
\varphi(x)=\phi(x)
=
c_\kappa^{-1}2^{\mathbf N/2}e^{-|x|^2}.
\]
The Gaussian satisfies both sets of hypotheses above, and
\[
\varphi_y(x)=h_{y^2/2}(x),
\qquad
(\varphi_y*_\kappa f)(x)=H_{y^2/2}f(x).
\]
Consequently, for this choice of $\varphi$,
\[
S_{\Gamma,\varphi}f(x)
=\bigg(\int_0^\infty y|\nabla_\Gamma H_{y^2/2}f(x)|^2\,dy\bigg)^{\frac12},
\qquad
S_{y,\varphi}f(x)
=\bigg(\int_0^\infty y|\partial_y H_{y^2/2}f(x)|^2\,dy\bigg)^{\frac12}.
\]
Now, using \eqref{eq:dunkl-poisson-subordination} and Minkowski's inequality yields
\[
G_\Gamma f(x)
=\left\|\nabla_\Gamma U_f(x,\cdot)\right\|_{L^2(y\,dy)}
\leq\frac1{\sqrt\pi}\int_0^\infty\frac{e^{-t}}{\sqrt t}
\left\|\nabla_\Gamma H_{\cdot^2/(2t)}f(x)\right\|_{L^2(y\,dy)}\,dt.
\]
For a fixed $t>0$, substituting $y=r\sqrt t$ gives
\begin{align*}
\left\|\nabla_\Gamma H_{\cdot^2/(2t)}f(x)\right\|_{L^2(y\,dy)}^2
=\int_0^\infty y\left|\nabla_\Gamma H_{y^2/(2t)}f(x)\right|^2\,dy
=t\left[S_{\Gamma,\varphi}f(x)\right]^2.
\end{align*}
Therefore
\[
G_\Gamma f(x)\leq\frac{S_{\Gamma,\varphi}f(x)}{\sqrt\pi}
\int_0^\infty e^{-t}\,dt
=\frac1{\sqrt\pi}S_{\Gamma,\varphi}f(x).
\]
Since $S_{\Gamma,\varphi}f$ is bounded on $L^p(dw)$ by \cref{T: Dzhe}, $G_\Gamma f$ is also bounded on $L^p(dw)$. For the vertical derivative, set
$\Phi(r):=H_{r^2/2}f(x)$. Then
\[
H_{y^2/(2t)}f(x)=\Phi(y/\sqrt t),
\qquad
\partial_yH_{y^2/(2t)}f(x)
=
\frac1{\sqrt t}\Phi'(y/\sqrt t).
\]
The same change of variables gives
\[
\left\|\partial_yH_{\cdot^2/(2t)}f(x)\right\|_{L^2(y\,dy)}
=S_{y,\varphi}f(x),
\]
and hence
\[
G_yf(x)
\leq\frac{S_{y,\varphi}f(x)}{\sqrt\pi}
\int_0^\infty e^{-t}t^{-1/2}\,dt
=S_{y,\varphi}f(x).
\]
Again by \cref{T: Dzhe} the result follows.
\end{proof}

\section{The probabilistic set-up}
\label{prob set up}
Throughout this section,
$(\Omega,\Sigma,(\Sigma_t)_{t\geq0},\mathbb P)$ denotes a filtered
probability space satisfying the usual conditions. In the following, we will introduce the Markov process generated by
$\frac12\Delta_\kappa$. Since the Dunkl Laplacian is not a local operator, the associated Markov process is not pathwise continuous but is a c\`adl\`ag process. Therefore, we first recall the optional and predictable
variations for c\`adl\`ag martingales and the compensators for their jump measures. We refer to \cite{protter,cohell} for the standard martingale theory mentioned below. From now on, for a c\`adl\`ag process $X$, we write
\[
\Delta X_t:=X_t-X_{t-},
\qquad t>0.
\]

\begin{definition}
Let $X$ be a real-valued c\`adl\`ag local martingale, and let $X^c$ denote its continuous local martingale part. Its optional quadratic variation is
\[
[X]_t:=[X^c]_t+\sum_{0<s\leq t}|\Delta X_s|^2,
\]
where $[X^c]$ is the usual quadratic variation of the continuous local martingale $X^c$. The process $[X]$ is adapted, increasing and c\`adl\`ag, with $[X]_0=0$, and
\[
X_t^2-X_0^2-[X]_t
\]
is a local martingale. If $X$ and $Y$ are real-valued c\`adl\`ag local
martingales, their optional quadratic covariation is defined by
\[
[X,Y]
:=
\frac14\bigl([X+Y]-[X-Y]\bigr).
\]
\end{definition}

If $X^c$ and $Y^c$ denote the continuous local martingale parts of
$X$ and $Y$, then
\begin{equation}
[X,Y]_t
=
[X^c,Y^c]_t
+
\sum_{0<s\leq t}
\Delta X_s\,\Delta Y_s.
\label{eq:optional-covariation-decomposition}
\end{equation}
If $X=(X^1,\ldots,X^d)$ is vector-valued, we define
$[X]_t:=\sum_{j=1}^d[X^j]_t$.

\begin{definition}
Let $X$ be a locally square-integrable c\`adl\`ag martingale. Its
predictable quadratic variation, or predictable bracket, is the unique
\emph{predictable}, increasing, c\`adl\`ag process $\langle X\rangle$, with
$\langle X\rangle_0=0$, such that
\[
X_t^2-X_0^2-\langle X\rangle_t
\]
is a local martingale. If $X$ and $Y$ are locally square-integrable
martingales, their predictable covariation is defined by
\[
\langle X,Y\rangle
:=
\frac14
\bigl(
\langle X+Y\rangle
-
\langle X-Y\rangle
\bigr).
\]
\end{definition}
For continuous local martingales, we have $\langle X,Y\rangle=[X,Y]$.
For discontinuous martingales, the two processes generally differ:
$[X,Y]$ contains the realized jump products, whereas
$\langle X,Y\rangle$ contains their predictable compensators.

\begin{remark}
Throughout the paper, we frequently use the stochastic differential
notation
\[
d[X,Y]_t
\qquad\text{and}\qquad
d\langle X,Y\rangle_t.
\]
These are understood as the random signed measures induced by $[X,Y]$ and $\langle X,Y\rangle$, respectively. In particular, $d\langle X,Y\rangle$ is the predictable
compensator of $d[X,Y]$, that is, for every bounded predictable
process $H$ for which the integrals are integrable,
\begin{equation}
\mathbb E\left[
\int_0^t H_s\,d[X,Y]_s
\right]
=
\mathbb E\left[
\int_0^t H_s\,d\langle X,Y\rangle_s
\right].
\end{equation}
\end{remark}

\subsection{Jump measures and their compensators}

Let $X$ be an $\R^d$-valued c\`adl\`ag adapted process. Its jump
measure $\mu^X$ is the integer-valued random measure on
$(0,\infty)\times(\R^d\setminus\{0\})$ that counts the realized jumps
of $X$:
\begin{equation*}
\mu^X(\omega;dt,dz)
:=
\sum_{s>0}
\ind_{\{\Delta X_s(\omega)\neq0\}}
\delta_{(s,\Delta X_s(\omega))}(dt,dz).
\end{equation*}
For every nonnegative measurable function $h$, integration with
respect to $\mu^X$ reduces to a pathwise sum over the jump times, that is, for every $t>0$,
\begin{equation*}
\int_{(0,t]}
\int_{\R^d\setminus\{0\}}
h_s(z)\,\mu^X(ds,dz)
=
\sum_{0<s\leq t}h_s(\Delta X_s),
\end{equation*}
with the convention $h_s(0)=0$.
Following the standard convention \cite{protter,cohell}, the same
notation is used for signed measurable $h$ whenever the corresponding
sum is absolutely convergent almost surely.
In the following we will also use the predictable compensator, or dual predictable projection, of
$\mu^X$, which is the predictable random measure $\nu^X$ characterized by
\begin{equation*}
\mathbb E\left[
\int_{(0,t]}
\int_{\R^d\setminus\{0\}}
h_s(z)\,\mu^X(ds,dz)
\right]
=
\mathbb E\left[
\int_{(0,t]}
\int_{\R^d\setminus\{0\}}
h_s(z)\,\nu^X(ds,dz)
\right],
\qquad t\geq0,
\end{equation*}
for every nonnegative predictable integrand $h$.

\subsection{Differential subordination and orthogonality}

The distinction between optional and predictable brackets is important
for the martingale inequalities used later.

\begin{definition}
Let $X$ and $Y$ be real or finite-dimensional Hilbert-valued c\`adl\`ag local martingales. We say that
$Y$ is differentially subordinate to $X$, and write $Y\ll X$, if
$|Y_0|\leq|X_0|$ and
$[X]_t-[Y]_t$
is nonnegative and nondecreasing as a function of $t$.
\end{definition}

As observed in \cite{W95}, provided that $|Y_0|\leq|X_0|$, the condition
$Y\ll X$ is equivalent to
\begin{enumerate}
\item $[X^c]-[Y^c]$ is nonnegative and nondecreasing;
\item $|\Delta Y_t|\leq|\Delta X_t|$ for every $t>0$, almost surely.
\end{enumerate}

\begin{definition}
Two real-valued c\`adl\`ag local martingales $X$ and $Y$ are called
orthogonal, and we write $X\perp Y$, if $[X,Y]\equiv0$.
\end{definition}

By \eqref{eq:optional-covariation-decomposition},
orthogonality is equivalent almost surely to
\[
[X^c,Y^c]\equiv0,
\qquad
\Delta X_t\,\Delta Y_t=0
\quad\text{for every }t>0.
\]

For $1<p<\infty$, write
\[
p^*:=\max\left\{p,\frac p{p-1}\right\},
\qquad
\|X\|_p:=\sup_{t\geq0}\|X_t\|_{L^p}.
\]

\begin{theorem}[\cite{W95,BW96}]
\label{thm:sharp-martingale-inequalities}
Let $X$ and $Y$ be real or finite-dimensional Hilbert-valued c\`adl\`ag local martingales and assume that
$Y\ll X$. Then
\[
\|Y\|_p
\leq
(p^*-1)\|X\|_p.
\]
If $X$ and $Y$ are real-valued and, in addition, $X\perp Y$, then
\[
\|Y\|_p
\leq
\cot\left(\frac{\pi}{2p^*}\right)\|X\|_p.
\]
The constants are sharp.
\end{theorem}

\begin{remark}
Differential subordination and orthogonality in
\cref{thm:sharp-martingale-inequalities} are conditions on the
optional brackets. For discontinuous martingales, the weaker condition
\[
\langle X,Y\rangle\equiv0
\]
will be called predictable orthogonality. Similarly, one may say that
$Y$ is predictably subordinate to $X$ if $|Y_0|\leq|X_0|$ and
$\langle X\rangle-\langle Y\rangle$
is nonnegative and nondecreasing. These predictable conditions do not,
in general, imply their optional counterparts. Consequently, the sharp
inequalities above cannot be invoked from these conditions alone.
\end{remark}

\subsection{The Dunkl process}

We now record the probabilistic features of the Dunkl setting that will
be used later. The Dunkl process $X=(X_t)_{t\geq0}$ is the c\`adl\`ag Feller process such that, for every Borel set $A\subset\R^n$,
\[
\mathbb P_x(X_t\in A)
=
\int_A h_t(x,y)\,dw(y),
\]
that is, its transition semigroup is
\[
H_t=e^{t\Delta_\kappa/2}
\]
and the infinitesimal generator is
\[
\mathcal L_\kappa:=\frac12\Delta_\kappa.
\]

We work with the usual augmentation of the natural filtration of $X$,
namely
\[
\Sigma_t^X
:=
\bigcap_{u>t}
\left(
\sigma(X_s:0\leq s\leq u)
\vee\mathcal N
\right),
\]
where $\mathcal N$ denotes the collection of null sets. The difference term
in $\Delta_\kappa$ produces the nonzero L\'evy kernel described below
and therefore accounts for the jumps of the process.

\subsubsection{The L\'evy kernel, reflection jumps and martingale decomposition}

The following description of the L\'evy system and the reflection jumps
is taken from \cite{GallardoYor2006_Chaotic}. By \eqref{eq:dunkl-laplacian}, away from the reflection hyperplanes the
generator $\mathcal L_\kappa$ can be written as
\begin{align}
\mathcal L_\kappa f(x)
=
\frac12\Delta f(x)
+
\sum_{\alpha\in\RR_+}
\kappa(\alpha)
\frac{\partial_\alpha f(x)}{\langle\alpha,x\rangle}
+
\int_{\R^n}
\bigl(f(y)-f(x)\bigr)N_\kappa(x,dy),
\end{align}
where the L\'evy kernel written in terms of the post-jump position is
\begin{equation*}
N_\kappa(x,dy)
=
\sum_{\alpha\in\RR_+}
r_\alpha(x)\,\delta_{\sigma_\alpha(x)}(dy),
\qquad
r_\alpha(x)
:=
\frac{\kappa(\alpha)}{\langle\alpha,x\rangle^2}
\ind_{\{\langle\alpha,x\rangle\neq0\}},
\end{equation*}
with $r_\alpha=0$ on $H_\alpha$.
Since
$\sigma_\alpha x-x=-\langle\alpha,x\rangle\alpha$,
the corresponding kernel in terms of jump increments is
\begin{equation*}
\mathcal K_\kappa(x,dz)
=
\sum_{\alpha\in\RR_+}
r_\alpha(x)\,\delta_{-\langle\alpha,x\rangle\alpha}(dz).
\end{equation*}
If $\mu^X$ is the jump measure of the Dunkl process, its predictable
compensator is
\begin{equation*}
\nu^X(dt,dz)
=
\mathcal K_\kappa(X_{t-},dz)\,dt.
\end{equation*}
Consequently, for every nonnegative predictable function $h$,
\begin{equation}
\begin{split}
\mathbb E_x\left[
\sum_{0<s\leq t}
h_s(X_{s-},X_s)\ind_{\{\Delta X_s\neq0\}}
\right]
=
\mathbb E_x\left[
\int_0^t
\sum_{\alpha\in\RR_+^\kappa}
r_\alpha(X_{s-})h_s(X_{s-},\sigma_\alpha X_{s-})\,ds
\right].
\end{split}
\label{compensation formula}
\end{equation}

The support of $N_\kappa(x,\cdot)$ shows that every nonzero jump is a reflection jump: almost surely, if $\Delta X_s\neq0$, then there exists $\alpha\in\RR_+^\kappa$ such that
\[
X_s=\sigma_\alpha X_{s-},
\qquad
\Delta X_s=-\langle\alpha,X_{s-}\rangle\alpha.
\]
At a jump of type $\alpha$, define
\[
\Delta_\alpha f(x)
:=
f(\sigma_\alpha x)-f(x)
=
-\frac{\langle\alpha,x\rangle}{\sqrt{\kappa(\alpha)}}J^\alpha f(x).
\]
The jump intensity and the jump size then satisfy, away from $H_\alpha$, the identity
\[
r_\alpha(x)\Delta_\alpha f(x)\Delta_\alpha g(x)
=
J^\alpha f(x)J^\alpha g(x).
\]
The values on the reflection hyperplanes do not affect the time integrals below, since the transition densities imply that the process spends zero Lebesgue time there.
Notice that roots with zero multiplicity do not produce jumps. This is why we assumed for simplicity that $\kappa(\alpha)>0$ for $\alpha\in\RR_+$.
The next theorem decomposes the Dunkl process into its continuous and
jump martingale coordinates. We also record the corresponding It\^o
formula.

\begin{theorem}[Martingale decomposition, \cite{GallardoYor2006_Chaotic}]
Let $x\in\R^n\setminus\{0\}$. Under $\mathbb P_x$, the Dunkl process
admits the decomposition
\begin{equation}
X_t
=
x+B_t
+
\sum_{\alpha\in\RR_+}
\sqrt{\kappa(\alpha)}\,M_t^\alpha\alpha,
\end{equation}
where $B$ is a standard $n$-dimensional Brownian motion and
$(M^\alpha)_{\alpha\in\RR_+}$ is a family of pairwise
orthogonal normal martingales. More precisely,
\[
\langle M^\alpha\rangle_t=t,
\qquad
[M^\alpha,M^\beta]\equiv0
\quad(\alpha\neq\beta),
\]
and $[B^j,M^\alpha]\equiv0$
for every $j$ and $\alpha$. The martingales $M^\alpha$ have paths of
finite variation and are given by
\begin{align*}
M_t^\alpha
=
\sum_{0<s\leq t}
\frac{-\langle\alpha,X_{s-}\rangle}{\sqrt{\kappa(\alpha)}}
\ind_{\{X_s=\sigma_\alpha X_{s-}\neq X_{s-}\}}
+
\int_0^t\frac{\sqrt{\kappa(\alpha)}}{\langle\alpha,X_{s-}\rangle}\,ds.
\end{align*}
\end{theorem}

\begin{theorem}[\cite{GallardoYor2006_Chaotic}]
Let $f\in C_c^{2,1}(\R^n\times[0,\infty))$. Then
\begin{align*}
f(X_t,t)
&=
f(X_0,0)
+
\int_0^t\langle\nabla_xf(X_{s-},s),dB_s\rangle\\
&\quad+
\sum_{\alpha\in\RR_+}\int_0^t J^\alpha f(X_{s-},s)\,dM_s^\alpha
+
\int_0^t\left(\partial_sf+\frac12\Delta_\kappa f\right)(X_{s-},s)\,ds.
\end{align*}
\end{theorem}

\subsection{The upper half-space process}

We now construct the process used in the probabilistic representation
of the Poisson extension. Fix $\lambda>0$ and $x\in\R^n$. Let
$X$ be the Dunkl process started at $x$ and $Y^\lambda$ be a
one-dimensional Brownian motion started at $\lambda$, independent of
the Dunkl process. We write the $(n+1)$-dimensional state process as
\[
Z^\lambda:=(X,Y^\lambda).
\]
Thus $Z^\lambda$ starts at $(x,\lambda)$ on the horizontal hyperplane
$\R^n\times\{\lambda\}$. Under $\mathbb P_{(x,\lambda)}$, the process
$Z^\lambda$ is a Feller process with generator
\[
\mathcal L_Z
=
\frac12\left(\Delta_\kappa+\frac{\partial^2}{\partial y^2}\right).
\]
The driving martingale is the $(n+m+1)$-dimensional process
\[
\mathbf Z_t^\lambda
:=
\bigl(B_t^1,\ldots,B_t^n,
M_t^{\alpha_1},\ldots,M_t^{\alpha_m},Y_t^\lambda\bigr)^\top.
\]
We distinguish throughout between the state process $Z^\lambda$, at which the Poisson extension is evaluated, and the driving martingale $\mathbf Z^\lambda$, against which the stochastic integrals are taken.

For completeness, these processes may be realized on the canonical path
space
\[
\Omega
:=
\mathbb D([0,\infty),\R^n)\times C([0,\infty),\R),
\]
where $\mathbb D([0,\infty),\R^n)$ is the Skorokhod space of c\`adl\`ag functions.
Since $dw$ is not a probability measure, introduce the $\sigma$-finite
measure on the path space
\[
\mathbb P^\lambda(A)
:=
\int_{\R^n}\mathbb P_{(x,\lambda)}(A)\,dw(x)
\]
and, for every $\mathbb P^\lambda$-integrable random variable $\xi$,
set
\begin{equation*}
\mathbb E^\lambda[\xi]
:=
\int_{\R^n}\mathbb E^{(x,\lambda)}[\xi]\,dw(x).
\end{equation*}
Let $\tau^\lambda:=\inf\{t\geq0:Y_t^\lambda=0\}$
be the exit time of $Z^\lambda$ from
$\R^n\times(0,\infty)$. Since $\tau^\lambda$ depends only on
$Y^\lambda$, it is independent of $X$.

\begin{proposition}
\label[proposition]{lambda exp}
For every $\lambda>0$ and for every function $f\in L^1(dw)$ we have
\[
\mathbb E^\lambda[f(X_{\tau^\lambda})]
=
\int_{\R^n}f(x)\,dw(x).
\]
Moreover, if $\varphi\geq0$ is measurable, then
\begin{equation}
\mathbb E^\lambda\left[
\int_0^{\tau^\lambda}\varphi(Z_s^\lambda)\,ds
\right]
=
\int_0^\infty\int_{\R^n}
2(y\wedge\lambda)\varphi(x,y)\,dw(x)\,dy.
\label{eq:occupation-identity}
\end{equation}
The same formula holds for signed $\varphi$ whenever the right-hand
side with $|\varphi|$ is finite.
\end{proposition}

\begin{proof}
Let $\mathbb E_\lambda^Y$ denote expectation with respect to the
one-dimensional Brownian motion started at $\lambda$. The heat
semigroup is conservative and symmetric with respect to $dw$. Hence
\[
\int_{\R^n}H_tg(x)\,dw(x)
=
\int_{\R^n}g(x)\,dw(x)
\]
for every suitable $g$. By independence and conditioning on
$\tau^\lambda$,
\begin{align*}
\mathbb E^\lambda[f(X_{\tau^\lambda})]
=
\int_{\R^n}\mathbb E^{(x,\lambda)}[f(X_{\tau^\lambda})]\,dw(x)
=
\mathbb E_\lambda^Y\left[
\int_{\R^n}H_{\tau^\lambda}f(x)\,dw(x)
\right]
=
\int_{\R^n}f(x)\,dw(x).
\end{align*}

For the occupation identity, let
$q_s^{(0,\infty)}(\lambda,y)$ denote the transition density of a
one-dimensional Brownian motion killed at the origin. By
independence we have
\begin{align*}
\mathbb E^\lambda\left[
\int_0^{\tau^\lambda}\varphi(Z_s^\lambda)\,ds
\right]
&=
\int_0^\infty\int_{\R^n}
\mathbb E^{(x,\lambda)}\left[
\ind_{\{s<\tau^\lambda\}}\varphi(X_s,Y_s^\lambda)
\right]\,dw(x)\,ds\\
&=
\int_0^\infty\int_{\R^n}\int_{\R^n}\int_0^\infty
h_s(x,z)q_s^{(0,\infty)}(\lambda,y)\varphi(z,y)
\,dy\,dw(z)\,dw(x)\,ds.
\end{align*}
Using \eqref{E: dunkl heat conservative} and the Green-kernel identity
\[
\int_0^\infty q_s^{(0,\infty)}(\lambda,y)\,ds
=2(y\wedge\lambda),
\]
we obtain \eqref{eq:occupation-identity}.
\end{proof}

\subsubsection{The stopped Poisson martingale}
Recall that for $z=(x,y)\in\R^n\times(0,\infty)$, the augmented gradient vector is
\[
\widetilde\nabla U_f(z)
:=
\bigl(
\nabla_xU_f(z),
J^{\alpha_1}U_f(z),\ldots,J^{\alpha_m}U_f(z),
\partial_yU_f(z)
\bigr)^\top.
\]
By the properties of the martingales
$M^{\alpha_\ell}$, the orthogonality of $B$ and
$M^{\alpha_\ell}$, and the independence of $Y^\lambda$ from the Dunkl
process, the coordinates of $\mathbf Z^\lambda$ are predictably
orthonormal:
\begin{equation*}
d\left\langle
\mathbf Z^{\lambda,r},\mathbf Z^{\lambda,s}
\right\rangle_t
=\delta_{rs}\,dt,
\qquad
1\leq r,s\leq n+m+1.
\end{equation*}
This is a statement about predictable brackets. The optional brackets
of the coordinates $M^{\alpha_\ell}$ still contain their realized
reflection jumps. \\
Let $f\in\mathcal S(\R^n)$. Since $U_f$ is Dunkl harmonic, we have
\[
\partial_y^2U_f+\Delta_\kappa U_f=0.
\]
It\^o's formula then yields the following representation:
\begin{equation}
f(X_{\tau^\lambda})
=
U_f(Z_0^\lambda)
+
\int_0^{\tau^\lambda}
\left\langle\widetilde\nabla U_f(Z_{s-}^\lambda),d\mathbf Z_s^\lambda\right\rangle.
\label{E: ito}
\end{equation}

\section{Dunkl Martingale Transforms}
\label{sect 4}
Let
$A:\R^n\times(0,\infty)\longrightarrow\mathcal M_{n+m+1}(\R)$
be a bounded measurable matrix-valued function and
\[
\|A\|_\infty
:=
\operatorname*{ess\,sup}_{(x,y)\in\R^n\times(0,\infty)}
\|A(x,y)\|_{\mathrm{op}}.
\]
For $f\in\mathcal S(\R^n)$, define the stopped martingale transform
\begin{equation}
(A\ast f)_t^\lambda
:=
\int_0^{t\wedge\tau^\lambda}
\left\langle
A(Z_{s-}^\lambda)\widetilde\nabla U_f(Z_{s-}^\lambda),
d\mathbf Z_s^\lambda
\right\rangle.
\end{equation}
The exit-distribution identity in \cref{lambda exp} shows that the distribution of $X_{\tau^\lambda}$ under the $\sigma$-finite measure $\mathbb P^\lambda$ is $dw$. We may therefore define the truncated Dunkl martingale projection as
\begin{equation}
\mathcal T_A^\lambda f(x)
:=
\mathbb E^\lambda\left[
(A\ast f)_{\tau^\lambda}^\lambda
\,\middle|\,
X_{\tau^\lambda}=x
\right].
\label{eq:truncated-dunkl-projection}
\end{equation}
In this section, the preliminary constants $C_p$ may depend on the fixed Dunkl structure. They are not used in the quantitative estimates of \cref{single riesz,vector}.

\begin{theorem}
\label{thm:dunkl-martingale-projection}
Let $A$ be as above and let $f,g\in\mathcal S(\R^n)$. Then
\begin{equation}
\langle\mathcal T_A^\lambda f,g\rangle_{L^2(dw)}
=
2\int_{\R^n}\int_0^\infty
(y\wedge\lambda)
\left\langle\widetilde\nabla U_g(x,y),
A(x,y)\widetilde\nabla U_f(x,y)\right\rangle
\,dy\,dw(x).
\label{eq:bilinear_predictable}
\end{equation}
For every $1<p<\infty$, the operator $\mathcal T_A^\lambda$ extends boundedly to $L^p(dw)$ and
\[
\sup_{\lambda>0}
\|\mathcal T_A^\lambda\|_{L^p(dw)\to L^p(dw)}
\leq C_p\|A\|_\infty.
\]
Moreover, there exists a bounded operator $\mathcal T_A:L^p(dw)\to L^p(dw)$ such that
\begin{equation}
\langle\mathcal T_Af,g\rangle
=
2\int_{\R^n}\int_0^\infty
 y\left\langle\widetilde\nabla U_g(x,y),
A(x,y)\widetilde\nabla U_f(x,y)\right\rangle
\,dy\,dw(x)
\label{eq:limit-TA-bilinear-form}
\end{equation}
for $f,g\in\mathcal S(\R^n)$, and
\[
\mathcal T_A^\lambda f\longrightarrow\mathcal T_Af
\qquad\text{in }L^p(dw).
\]
\end{theorem}

\begin{proof}
Fix $\lambda>0$ and suppress the superscript $\lambda$ on $Z^\lambda$ and $\mathbf Z^\lambda$, and write $\tau=\tau^\lambda$. By predictable orthonormality and the occupation identity,
\begin{align*}
\mathbb E^\lambda\left[
\int_0^\tau\left|A(Z_{s-})\widetilde\nabla U_f(Z_{s-})\right|^2ds
\right]
=
2\int_{\R^n}\int_0^\infty
(y\wedge\lambda)\left|A(x,y)\widetilde\nabla U_f(x,y)\right|^2
\,dy\,dw(x),
\end{align*}
and by \cref{L:augmented-square-function-bound} the right-hand side is finite.
The same argument applies to the stochastic integral in $g(X_\tau)$. By the exit-distribution identity and \eqref{eq:truncated-dunkl-projection},
\begin{align*}
\langle\mathcal T_A^\lambda f,g\rangle
&=
\mathbb E^\lambda\left[g(X_\tau)\mathcal T_A^\lambda f(X_\tau)\right]
=
\mathbb E^\lambda\left[g(X_\tau)(A\ast f)_\tau^\lambda\right].
\end{align*}
Using \eqref{E: ito}, write
\[
g(X_\tau)
=
U_g(Z_0)
+
\int_0^\tau\left\langle\widetilde\nabla U_g(Z_{s-}),d\mathbf Z_s\right\rangle.
\]
For $dw$-almost every initial point $x$, the transform $(A\ast f)^\lambda$ is a square-integrable martingale starting from zero, hence the term containing $U_g(Z_0)$ has zero expectation. Since
\begin{align*}
&d\left\langle
\int_0^\cdot\left\langle\widetilde\nabla U_g(Z_{s-}),d\mathbf Z_s\right\rangle,
\int_0^\cdot\left\langle A(Z_{s-})\widetilde\nabla U_f(Z_{s-}),d\mathbf Z_s\right\rangle
\right\rangle_t\\
&\hspace{4em}=
\left\langle\widetilde\nabla U_g(Z_{t-}),A(Z_{t-})\widetilde\nabla U_f(Z_{t-})\right\rangle dt,
\end{align*}
taking expectations and then applying \eqref{eq:occupation-identity} proves \eqref{eq:bilinear_predictable}. Then, Cauchy--Schwarz in the $y$ variable, H\"older's inequality and \cref{L:augmented-square-function-bound} yield
\begin{align*}
|\langle\mathcal T_A^\lambda f,g\rangle|
\leq C_p\|A\|_\infty\|f\|_{L^p(dw)}\|g\|_{L^{p'}(dw)}
\end{align*}
and establish the boundedness of the form in \eqref{eq:limit-TA-bilinear-form} at the same time. It remains to prove strong convergence. For $f,g\in\mathcal S(\R^n)$, subtracting \eqref{eq:bilinear_predictable} from \eqref{eq:limit-TA-bilinear-form} gives
\[
\langle(\mathcal T_A-\mathcal T_A^\lambda)f,g\rangle
=
2\int_{\R^n}\int_\lambda^\infty
(y-\lambda)\left\langle\widetilde\nabla U_g,A\widetilde\nabla U_f\right\rangle
\,dy\,dw.
\]
Define
\[
\widetilde G_{\Gamma,\lambda}^{\mathrm{tail}}f(x)
:=
\left(\int_\lambda^\infty
(y-\lambda)|\widetilde\nabla U_f(x,y)|^2\,dy\right)^{1/2}.
\]
Then
\[
\widetilde G_{\Gamma,\lambda}^{\mathrm{tail}}f\leq\widetilde G_\Gamma f,
\qquad
\widetilde G_{\Gamma,\lambda}^{\mathrm{tail}}f(x)\longrightarrow0
\quad\text{for a.e. }x\in\R^n.
\]
In particular, dominated convergence gives
$\|\widetilde G_{\Gamma,\lambda}^{\mathrm{tail}}f\|_{L^p}\to0$
as $\lambda\to\infty$, and
\begin{align*}
|\langle(\mathcal T_A-\mathcal T_A^\lambda)f,g\rangle|
&\leq
2\|A\|_\infty
\left\|\widetilde G_{\Gamma,\lambda}^{\mathrm{tail}}f\right\|_{L^p(dw)}
\|\widetilde G_\Gamma g\|_{L^{p'}(dw)}\\
&\leq
C_p\|A\|_\infty
\left\|\widetilde G_{\Gamma,\lambda}^{\mathrm{tail}}f\right\|_{L^p(dw)}
\|g\|_{L^{p'}(dw)}.
\end{align*}
Taking the supremum over $\|g\|_{p'}\leq1$ proves convergence for $f\in\mathcal S(\R^n)$. The uniform bounds for $\mathcal T_A^\lambda$ and density extend the result to $f\in L^p(dw)$.
\end{proof}

\subsection{Canonical lifts of Euclidean martingale transforms}

Suppose a Euclidean Fourier multiplier is obtained by projecting a space--time martingale transform with a constant matrix $E$. We now identify the corresponding matrix acting on the extended Dunkl coefficient space. Set
\[
\mathcal H_\kappa
:=\R^n\oplus\ell^2(\RR_+^\kappa)\oplus\R.
\]
The three components correspond to the spatial Brownian coordinates, the jump martingales, and the vertical Brownian coordinate. Thus
\[
\widetilde\nabla U_f
=
\bigl(\nabla_xU_f,(J^\alpha U_f)_{\alpha\in\RR_+^\kappa},\partial_yU_f\bigr)^\top.
\]

\begin{definition}
The horizontal Dunkl compression map is
\[
\mathsf S_\kappa:
\R^n\oplus\ell^2(\RR_+^\kappa)\longrightarrow\R^n,
\qquad
\mathsf S_\kappa(a,q)
:=a+\sum_{\alpha\in\RR_+^\kappa}\beta_\alpha q_\alpha,
\]
where $\beta_\alpha=\sqrt{\kappa(\alpha)}\alpha$. The full space--time compression is
\[
\mathcal C_\kappa:\mathcal H_\kappa\longrightarrow\R^n\oplus\R,
\qquad
\mathcal C_\kappa(a,q,r):=\bigl(\mathsf S_\kappa(a,q),r\bigr).
\]
\end{definition}

Defining
\[
\mathcal B_\kappa
:=\begin{bmatrix}\beta_{\alpha_1}&\cdots&\beta_{\alpha_m}\end{bmatrix}
\in\mathcal M_{n\times m}(\R),
\]
we remark that
\[
\mathcal C_\kappa
=
\begin{bmatrix}
I_n&\mathcal B_\kappa&\mathbf0_{n\times1}\\
\mathbf0_{1\times n}&\mathbf0_{1\times m}&1
\end{bmatrix},
\qquad
\mathcal C_\kappa\widetilde\nabla U_f
=\bigl(\nabla_\kappa U_f,\partial_yU_f\bigr)^\top.
\]

\begin{definition}
Let $E\in\mathcal M_{n+1}(\R)$ be a constant matrix defining a Euclidean space--time martingale transform. Its canonical Dunkl lift is
\begin{equation}
A_E:=\mathcal C_\kappa^*E\mathcal C_\kappa
\in\mathcal M_{n+m+1}(\R).
\label{eq:canonical-dunkl-lift}
\end{equation}
\end{definition}

\begin{proposition}
Let $A=A_E$ be as in \eqref{eq:canonical-dunkl-lift}. Then
\[
\langle\mathcal T_Af,g\rangle
=
2\int_0^\infty\int_{\R^n}
y\left\langle
\bigl(\nabla_\kappa U_g,\partial_yU_g\bigr),
E\bigl(\nabla_\kappa U_f,\partial_yU_f\bigr)
\right\rangle\,dw(x)\,dy.
\]
Moreover, $\mathcal T_A$ is a Dunkl Fourier multiplier with symbol
\begin{equation}
m_E(\xi)
=\frac{\zeta(\xi)^*E\zeta(\xi)}{2|\xi|^2},
\qquad
\zeta(\xi):=\begin{bmatrix}i\xi\\- |\xi|\end{bmatrix},
\qquad \xi\neq0.
\label{eq:explicit-lifted-symbol}
\end{equation}
\end{proposition}

\begin{proof}
The identity \eqref{eq:limit-TA-bilinear-form}, together with \eqref{eq:canonical-dunkl-lift} and the compression identity above, gives
\[
\langle\mathcal T_Af,g\rangle
=
2\int_0^\infty\int_{\R^n}
y\left\langle
\bigl(\nabla_\kappa U_g,\partial_yU_g\bigr),
E\bigl(\nabla_\kappa U_f,\partial_yU_f\bigr)
\right\rangle\,dw(x)\,dy.
\]
On the Dunkl-transform side,
\[
\mathcal F_\kappa\left(\nabla_\kappa U_f,\partial_yU_f\right)^\top(\xi,y)
=\zeta(\xi)e^{-y|\xi|}\mathcal F_\kappa f(\xi).
\]
Plancherel's theorem and Fubini's theorem therefore yield
\begin{align*}
\langle\mathcal T_Af,g\rangle
&=
2\int_0^\infty y\int_{\R^n}
e^{-2y|\xi|}\overline{\mathcal F_\kappa g(\xi)}
\zeta(\xi)^*E\zeta(\xi)\mathcal F_\kappa f(\xi)\,dw(\xi)\,dy.
\end{align*}
Since
\[
\int_0^\infty2ye^{-2y|\xi|}\,dy=\frac1{2|\xi|^2},
\qquad \xi\neq0,
\]
we obtain
\[
\langle\mathcal T_Af,g\rangle
=\int_{\R^n}m_E(\xi)\mathcal F_\kappa f(\xi)
\overline{\mathcal F_\kappa g(\xi)}\,dw(\xi),
\]
which proves \eqref{eq:explicit-lifted-symbol}.
\end{proof}

\subsubsection{First-order Dunkl Riesz transforms}

The $j$th Dunkl Riesz transform is defined by
\[
\mathcal F_\kappa(R_jf)(\xi)
=-i\frac{\xi_j}{|\xi|}\mathcal F_\kappa f(\xi),
\qquad \xi\neq0.
\]
In the Euclidean setting, the corresponding skew-symmetric space--time matrix is
\[
E_j:=E_{n+1,j}-E_{j,n+1}
=\begin{bmatrix}0&-e_j\\e_j^\top&0\end{bmatrix};
\]
see \cite{Ban1}. Its canonical lift is then
\begin{equation}
A_j:=\mathcal C_\kappa^*E_j\mathcal C_\kappa
=
\begin{bmatrix}
\mathbf0_{n\times n}&\mathbf0_{n\times m}&-e_j\\
\mathbf0_{m\times n}&\mathbf0_{m\times m}&-\beta^j\\
e_j^\top&(\beta^j)^\top&0
\end{bmatrix},
\end{equation}
where
$\beta^j:=\bigl(\beta_{\alpha_1}^j,\ldots,\beta_{\alpha_m}^j\bigr)^\top$.
Notice that, since
\[
dX_t^j=dB_t^j+\sum_{\alpha\in\RR_+^\kappa}\beta_\alpha^j\,dM_t^\alpha,
\qquad
T_jU_f=\partial_jU_f+\sum_{\alpha\in\RR_+^\kappa}\beta_\alpha^jJ^\alpha U_f,
\]
a direct computation shows that, up to the stopping time,
\[
d(A_j\ast f)_t
=\left\langle A_j\widetilde\nabla U_f(Z_{t-}),d\mathbf Z_t\right\rangle
=T_jU_f(Z_{t-})\,dY_t-\partial_yU_f(Z_{t-})\,dX_t^j.
\]
Therefore, as in the Euclidean case, the matrix $A_j$ swaps the $j$th Dunkl derivative and the vertical derivative.

\begin{corollary}
For $f,g\in\mathcal S(\R^n)$,
\[
\langle\mathcal T_{A_j}f,g\rangle=\langle R_jf,g\rangle.
\]
Equivalently, $\mathcal T_{A_j}=R_j$.
\end{corollary}

\begin{proof}
By \eqref{eq:explicit-lifted-symbol},
\[
\zeta(\xi)^*E_j\zeta(\xi)
=\begin{bmatrix}-i\xi^\top&-|\xi|\end{bmatrix}
\begin{bmatrix}|\xi|e_j\\i\xi_j\end{bmatrix}
=-2i\xi_j|\xi|.
\]
Therefore
\[
m_{E_j}(\xi)=\frac{-2i\xi_j|\xi|}{2|\xi|^2}
=-i\frac{\xi_j}{|\xi|}.
\]
\end{proof}

\begin{remark}
Let $L_j$ denote the lower-triangular part of $A_j$, so that $A_j=L_j-L_j^\top$. Applying \eqref{eq:explicit-lifted-symbol} to the corresponding lower-triangular Euclidean matrix gives
\[
m_{L_j}(\xi)=-\frac{i\xi_j}{2|\xi|}.
\]
Hence $\mathcal T_{L_j}=\frac12R_j$.
Likewise, if $U_j:=-L_j^\top$, then $\mathcal T_{U_j}=\frac12R_j$. Thus either triangular lift may be used to represent one half of the first-order Dunkl Riesz transform.
\end{remark}

\section{$G$-invariant functions and dimension-free bounds}
\label{sec:invariant-sharp-bounds}
For a general function $f\in\mathcal S(\R^n)$, the full matrix lift $A_j$ introduces reflection jumps driven by $\partial_yU_f\,dM^\alpha$ which generally destroy both differential subordination and optional orthogonality. When $f$ is $G$-invariant, however, its Poisson martingale becomes continuous and orthogonality is recovered, but differential subordination for the full lift may still fail because the transformed martingale can retain reflection jumps.

To obtain dimension-free estimates, we use the lower-triangular representation $L_j$ instead of $A_j$. On invariant inputs, the resulting transform is continuous and differentially subordinate to the Poisson martingale. This yields root-system-independent bounds for both the individual Riesz transforms and the Riesz vector. The argument does not recover the sharp cotangent constant for a single directional transform, which is currently an open question.

\subsection{Covariations and the failure of differential subordination}

Let $f\in\mathcal S(\R^n)$, and consider the uncentered stopped Poisson martingale
\[
W_t:=U_f(Z_{t\wedge\tau^\lambda}^\lambda)
=U_f(Z_0^\lambda)+(I\ast f)_t^\lambda
\]
and its transform $V_t:=(A_j\ast f)_t^\lambda$. In particular,
$W_0=P_\lambda f(X_0)$ and $W_{\tau^\lambda}=f(X_{\tau^\lambda})$.
In what follows, all martingales are stopped at $\tau^\lambda$, and all the identities are understood up to this stopping time. The stochastic differentials are
\begin{align*}
dW_t
&=
\sum_{k=1}^n\partial_kU_f(Z_{t-}^\lambda)\,dB_t^k
+
\sum_{\alpha\in\RR_+^\kappa}J^\alpha U_f(Z_{t-}^\lambda)\,dM_t^\alpha
+
\partial_yU_f(Z_{t-}^\lambda)\,dY_t^\lambda,\\
dV_t
&=
-\partial_yU_f(Z_{t-}^\lambda)\,dB_t^j
-
\sum_{\alpha\in\RR_+^\kappa}\beta_\alpha^j\partial_yU_f(Z_{t-}^\lambda)\,dM_t^\alpha
+
T_jU_f(Z_{t-}^\lambda)\,dY_t^\lambda.
\end{align*}

\begin{proposition}
\label[proposition]{prop:optional_brackets_general}
For every $f\in\mathcal S(\R^n)$,
\begin{align*}
d[W,W]_t
&=
\left(|\nabla_xU_f(Z_{t-}^\lambda)|^2+|\partial_yU_f(Z_{t-}^\lambda)|^2\right)dt
+
\sum_{\alpha\in\RR_+^\kappa}|J^\alpha U_f(Z_{t-}^\lambda)|^2\,d[M^\alpha,M^\alpha]_t,\\
d[V,V]_t
&=
\left(|T_jU_f(Z_{t-}^\lambda)|^2+|\partial_yU_f(Z_{t-}^\lambda)|^2\right)dt
+
|\partial_yU_f(Z_{t-}^\lambda)|^2
\sum_{\alpha\in\RR_+^\kappa}(\beta_\alpha^j)^2\,d[M^\alpha,M^\alpha]_t,\\
d[W,V]_t
&=
\partial_yU_f(Z_{t-}^\lambda)
\sum_{\alpha\in\RR_+^\kappa}\beta_\alpha^jJ^\alpha U_f(Z_{t-}^\lambda)
\bigl(dt-d[M^\alpha,M^\alpha]_t\bigr).
\end{align*}
\end{proposition}

\begin{proof}
Decompose $W=W_0+W^c+W^d$ and $V=V^c+V^d$ into their initial values, continuous Brownian parts and purely discontinuous parts. The continuous brackets satisfy
\begin{align*}
d[W^c,W^c]_t
&=\left(|\nabla_xU_f|^2+|\partial_yU_f|^2\right)(Z_{t-}^\lambda)\,dt,\\
d[V^c,V^c]_t
&=\left(|T_jU_f|^2+|\partial_yU_f|^2\right)(Z_{t-}^\lambda)\,dt,\\
d[W^c,V^c]_t
&=\partial_yU_f(Z_{t-}^\lambda)
\left(\sum_{\alpha\in\RR_+^\kappa}\beta_\alpha^jJ^\alpha U_f(Z_{t-}^\lambda)\right)dt.
\end{align*}
Pairwise optional orthogonality of the martingales $(M^\alpha)_\alpha$ gives
\begin{align*}
[W^d,W^d]_t
&=\sum_{\alpha\in\RR_+^\kappa}\int_0^t
|J^\alpha U_f(Z_{s-}^\lambda)|^2\,d[M^\alpha,M^\alpha]_s,\\
[V^d,V^d]_t
&=\sum_{\alpha\in\RR_+^\kappa}(\beta_\alpha^j)^2\int_0^t
|\partial_yU_f(Z_{s-}^\lambda)|^2\,d[M^\alpha,M^\alpha]_s,\\
[W^d,V^d]_t
&=-\sum_{\alpha\in\RR_+^\kappa}\int_0^t
\beta_\alpha^jJ^\alpha U_f(Z_{s-}^\lambda)
\partial_yU_f(Z_{s-}^\lambda)\,d[M^\alpha,M^\alpha]_s.
\end{align*}
Adding the continuous and discontinuous contributions proves the result.
\end{proof}

\begin{remark}
By predictable orthonormality of the martingale basis and skew-symmetry,
\[
d\langle W,W\rangle_t
=|\widetilde\nabla U_f(Z_{t-}^\lambda)|^2\,dt,
\qquad
d\langle V,V\rangle_t
\leq\|A_j\|_{\mathrm{op}}^2\,d\langle W,W\rangle_t,
\]
and
\[
d\langle W,V\rangle_t
=\left\langle\widetilde\nabla U_f(Z_{t-}^\lambda),
A_j\widetilde\nabla U_f(Z_{t-}^\lambda)\right\rangle dt=0.
\]
Thus $V/\|A_j\|_{\mathrm{op}}$ is predictably subordinate to $W$, while $V$ is predictably orthogonal to $W$. These predictable properties do not imply the corresponding optional ones. Indeed, at a reflection jump of type $\alpha$,
\[
|\Delta V_s|
=|\beta_\alpha^j\partial_yU_f(Z_{s-}^\lambda)|\,|\Delta M_s^\alpha|,
\qquad
|\Delta W_s|
=|J^\alpha U_f(Z_{s-}^\lambda)|\,|\Delta M_s^\alpha|,
\]
and there is no uniform pointwise comparison between these two quantities.
\end{remark}

\subsection{Dimension-free bounds via the triangular representation}
Suppose now that $f$ is $G$-invariant. Then $J^\alpha U_f\equiv0$ and $T_jU_f=\partial_jU_f$. It follows from \cref{prop:optional_brackets_general} that $W$ is continuous and that $[W,V]\equiv0$. Nevertheless, at a reflection jump of type $\alpha$,
\[
\Delta W_s=0,
\qquad
\Delta V_s=-\beta_\alpha^j\partial_yU_f(Z_{s-}^\lambda)\Delta M_s^\alpha.
\]
Thus differential subordination for the full lift may still fail. However, the augmented gradient takes values in the reflection-free subspace
\[
\mathcal H_{\mathrm{inv}}
:=\left\{(a,0,r)\in\mathcal H_\kappa:a\in\R^n,\ r\in\R\right\}.
\]
Consider the lower-triangular part of $A_j$:
\[
L_j
=\begin{bmatrix}
\mathbf0_{n\times n}&\mathbf0_{n\times m}&\mathbf0_{n\times1}\\
\mathbf0_{m\times n}&\mathbf0_{m\times m}&\mathbf0_{m\times1}\\
e_j^\top&(\beta^j)^\top&0
\end{bmatrix}.
\]
\begin{proposition}
\label[proposition]{prop:invariant_martingale_structure}
The matrix $L_j$ is a contraction on $\mathcal H_{\mathrm{inv}}$. Therefore, if $f\in\mathcal S(\R^n)$ is $G$-invariant, the scalar martingale transforms
\[
V_t^j:=(L_j\ast f)_t^\lambda,
\qquad j=1,\ldots,n,
\]
and the vector martingale
\[
\mathbf V_t^\lambda
:=\bigl((L_1\ast f)_t^\lambda,\ldots,(L_n\ast f)_t^\lambda\bigr)
\]
are both differentially subordinate to $W$.
\end{proposition}

\begin{proof}
If $z=(a,0,r)\in\mathcal H_{\mathrm{inv}}$, then
\[
L_j(a,0,r)=(0,0,a_j),
\qquad
|L_j(a,0,r)|=|a_j|\leq\sqrt{|a|^2+r^2}=|z|_{\mathcal H_\kappa}.
\]
Hence $\|L_j|_{\mathcal H_{\mathrm{inv}}}\|_{\mathrm{op}}\leq1$. In addition,
\[
dV_t^j=\partial_jU_f(Z_{t-}^\lambda)\,dY_t^\lambda,
\qquad
d\mathbf V_t^\lambda=\nabla_xU_f(Z_{t-}^\lambda)\,dY_t^\lambda.
\]
Therefore
\[
d[V^j,V^j]_t
\leq d[\mathbf V^\lambda]_t
=|\nabla_xU_f(Z_{t-}^\lambda)|^2\,dt
\leq d[W,W]_t.
\]
Since $V_0^j=0$ and $\mathbf V_0^\lambda=0$, the initial-value conditions also hold, which proves both statements.
\end{proof}

\begin{corollary}[Dimension-free bounds for invariant functions]
\label[corollary]{cor:invariant-bounds}
For every $1<p<\infty$ and every $G$-invariant function $f\in L^p(dw)$,
\[
\|\mathbf Rf\|_{L^p(dw;\ell_n^2)}
\leq2(p^*-1)\|f\|_{L^p(dw)}.
\]
In particular, for each $1\leq j\leq n$,
\[
\|R_jf\|_{L^p(dw)}\leq2(p^*-1)\|f\|_{L^p(dw)}.
\]
\end{corollary}

\begin{proof}
First assume that $f\in\mathcal S(\R^n)$ is $G$-invariant. Once differential subordination has been established, the passage to the operator estimate is standard in harmonic analysis. By conditional Jensen's inequality and \cref{thm:sharp-martingale-inequalities}, using \cref{prop:invariant_martingale_structure} under each initial law and then integrating in $x$, we obtain
\begin{align*}
\left\|\bigl(\mathcal T_{L_1}^\lambda f,\ldots,\mathcal T_{L_n}^\lambda f\bigr)\right\|_{L^p(dw;\ell_n^2)}
&\leq\|\mathbf V_{\tau^\lambda}^\lambda\|_{L^p(\mathbb P^\lambda;\ell_n^2)}\\
&\leq(p^*-1)\|W_{\tau^\lambda}\|_{L^p(\mathbb P^\lambda)}\\
&=(p^*-1)\|f\|_{L^p(dw)},
\end{align*}
where the last equality follows from $W_{\tau^\lambda}=f(X_{\tau^\lambda})$ and \cref{lambda exp}.
Taking $\lambda\to\infty$, using the strong convergence of the truncated projections, and recalling that $\mathcal T_{L_j}=\frac12R_j$, gives the vector estimate. The scalar estimate follows by taking one coordinate. Finally, the estimate extends to every $G$-invariant $f\in L^p(dw)$ by the density of $G$-invariant Schwartz functions in the invariant subspace of $L^p(dw)$.
\end{proof}

\subsection{The one-dimensional case}

In dimension one, $G$-invariant functions are precisely the even functions. The preceding corollary and duality therefore give a multiplicity-independent bound for the Dunkl Hilbert transform on all inputs.

\begin{corollary}
\label[corollary]{cor:rank-one-bound}
Let $n=1$ and $H_\kappa:=R_1$. For every $1<p<\infty$ and every $f\in L^p(dw)$,
\[
\|H_\kappa f\|_{L^p(dw)}
\leq4(p^*-1)\|f\|_{L^p(dw)},
\]
uniformly in $\kappa\geq0$.
\end{corollary}

\begin{proof}
The measure $dw$ is even in the one dimensional case. Also, the multiplier definition gives
\[
H_\kappa\bigl(f(-\cdot)\bigr)=-(H_\kappa f)(-\cdot),
\qquad
H_\kappa^*=-H_\kappa.
\]
Thus $H_\kappa$ interchanges even and odd functions. By \cref{cor:invariant-bounds}, for every even $g \in L^p$
\[
\|H_\kappa g\|_p\leq2(p^*-1)\|g\|_p.
\]
If $h$ is odd, then $H_\kappa h$ is even, so duality may be restricted to even test functions. Consequently,
\begin{align*}
\|H_\kappa h\|_p
&=\sup_{\substack{g\ \mathrm{even}\\\|g\|_{p'}\leq1}}
|\langle H_\kappa h,g\rangle|
=\sup_{\substack{g\ \mathrm{even}\\\|g\|_{p'}\leq1}}
|\langle h,H_\kappa g\rangle| \leq2(p^*-1)\|h\|_p.
\end{align*}
For arbitrary $f$, write
\[
f=f_{\mathrm e}+f_{\mathrm o},
\qquad
f_{\mathrm e}(x):=\frac{f(x)+f(-x)}2,
\qquad
f_{\mathrm o}(x):=\frac{f(x)-f(-x)}2.
\]
Both projections are contractions on $L^p(dw)$. Applying the even and odd estimates gives
\[
\|H_\kappa f\|_p
\leq2(p^*-1)\bigl(\|f_{\mathrm e}\|_p+\|f_{\mathrm o}\|_p\bigr)
\leq4(p^*-1)\|f\|_p.
\]
\end{proof}

\section{Bellman function estimates for non-invariant functions}

In this section we prove estimates for the Dunkl Riesz transforms without
assuming that the input is $G$-invariant. We still use the matrix
$L_j$ introduced in the previous section, yet differential
subordination to the Poisson martingale fails for non-invariant functions. We therefore apply It\^o's formula
directly to Burkholder's function and study both parts of its
drift term. The continuous contribution may be positive, but its defect is measured
by the same coefficient as the negative contribution produced by the
reflection jumps. After passing from the jump sum to its predictable
compensator through the L\'evy kernel, the two terms
cancel. We restate \cref{single riesz,vector} for general inputs. Define
\begin{equation}
\theta_j
:=\|A_j\|_{\mathrm{op}}
=\left(1+\sum_{\alpha\in\RR_+}(\beta_\alpha^j)^2\right)^{1/2}.
\end{equation}

\begin{theorem}
\label{Riesz}
For every $1<p<\infty$ and every $j\in\{1,\ldots,n\}$,
\[
\|R_jf\|_{L^p(dw)}
\leq2(p^*-1)\theta_j\|f\|_{L^p(dw)}.
\]
In particular, since $|\alpha|^2=2$,
\[
\|R_jf\|_{L^p(dw)}
\leq2(p^*-1)\sqrt{1+2\gamma}\,\|f\|_{L^p(dw)}.
\]
\end{theorem}

The same argument applies simultaneously to all the lower-triangular
representations when $p\geq2$. Recall the operator
\[
\mathsf S_\kappa:
\R^n\oplus\ell^2(\RR_+)\longrightarrow\R^n,
\qquad
\mathsf S_\kappa(a,q):=a+\sum_{\alpha\in\RR_+}\beta_\alpha q_\alpha,
\]
and set
\[
\Theta_\kappa:=\|\mathsf S_\kappa\|_{\mathrm{op}}
=\left\|I_{\R^n}
+\sum_{\alpha\in\RR_+}\kappa(\alpha)\,\alpha\otimes\alpha
\right\|_{\mathrm{op}}^{1/2}.
\]
Since $M_\kappa$ is positive semidefinite and $\operatorname{tr}M_\kappa=2\gamma$,
\[
\Theta_\kappa\leq\sqrt{1+2\gamma}.
\]

\begin{theorem}
\label{thm:Riesz-vector-noninvariant}
Let $p\geq2$. Then, for every $f\in L^p(dw)$,
\[
\|\mathbf Rf\|_{L^p(dw;\ell_n^2)}
\leq2(p-1)\Theta_\kappa\|f\|_{L^p(dw)}.
\]
Consequently,
\[
\|\mathbf Rf\|_{L^p(dw;\ell_n^2)}
\leq2(p-1)\sqrt{1+2\gamma}\,\|f\|_{L^p(dw)}.
\]
\end{theorem}

\begin{proof}[Proof of \cref{Riesz}]
We first assume that $p\geq2$ and $f\in\mathcal S(\R^n)$. Put
\[
C_p:=p\left(1-\frac1p\right)^{p-1},
\qquad D_p:=p-1,
\]
and consider the classical Burkholder function
\[
\mathcal B(w,v)
:=C_p\bigl(|v|-D_p|w|\bigr)\bigl(|w|+|v|\bigr)^{p-1}.
\]
We shall use the standard inequalities
\[
\mathcal B(w,v)\geq|v|^p-(p-1)^p|w|^p,
\qquad
\mathcal B(w,0)=-C_p(p-1)|w|^p\leq0.
\]

Fix $\lambda>0$, write $\tau:=\tau^\lambda$, and suppress the superscript
$\lambda$ from the processes. As in the preceding section, the source martingale is uncentered. Set
\[
W_t:=U_f(Z_{t\wedge\tau}),
\qquad
V_t:=\frac1{\theta_j}(L_j\ast f)_{t\wedge\tau}
=\frac1{\theta_j}\int_0^{t\wedge\tau}T_jU_f(Z_{s-})\,dY_s.
\]
Setting
\[
c_0:=\frac1{\theta_j},
\qquad
c_\alpha:=\frac{\beta_\alpha^j}{\theta_j},
\]
we have
\[
c_0^2+\sum_{\alpha\in\RR_+}c_\alpha^2=1,
\qquad
\frac{T_jU_f}{\theta_j}
=c_0\partial_jU_f+\sum_{\alpha\in\RR_+}c_\alpha J^\alpha U_f.
\]

The function $\mathcal B$ is not globally of class $C^2$. Let
$\rho\in C_c^\infty(\R^2)$ be nonnegative, even in each variable,
supported in the unit ball, and normalized by
\[
\int_{\R^2}\rho(z)\,dz=1.
\]
For $\varepsilon>0$, set
$\rho_\varepsilon(z):=\varepsilon^{-2}\rho(z/\varepsilon)$ and
$\mathcal B_\varepsilon:=\mathcal B*\rho_\varepsilon$.

Choose a localizing sequence $(\sigma_N)_{N\geq1}$ such that the
stochastic integrals below, stopped at $\sigma_N$, are true martingales
and the compensated jump integrands are integrable. We also stop when $|V|$ reaches $N$. Since $V$ is continuous and $|W|\leq\|f\|_\infty$, the stopped pair is bounded. Fix
$x\in\R^n$ and put $T:=t\wedge\tau\wedge\sigma_N$.
Applying It\^o's formula under $\mathbb P_{(x,\lambda)}$ gives
\begin{align*}
\mathbb E^{(x,\lambda)}\mathcal B_\varepsilon(W_T,V_T)
&=\mathcal B_\varepsilon(P_\lambda f(x),0)
+\mathbb E^{(x,\lambda)}\left[
\int_0^T\mathscr C_\varepsilon(s)\,ds
+\sum_{0<s\leq T}\mathscr J_\varepsilon(s)
\right],
\end{align*}
where
\begin{align*}
\mathscr C_\varepsilon(s)
&:={}
\frac12\partial_{ww}\mathcal B_\varepsilon(W_{s-},V_{s-})
\bigl(|\nabla_xU_f|^2+|\partial_yU_f|^2\bigr)(Z_{s-})\\
&\quad+
\partial_{wv}\mathcal B_\varepsilon(W_{s-},V_{s-})
\partial_yU_f(Z_{s-})\frac{T_jU_f(Z_{s-})}{\theta_j}\\
&\quad+
\frac12\partial_{vv}\mathcal B_\varepsilon(W_{s-},V_{s-})
\left|\frac{T_jU_f(Z_{s-})}{\theta_j}\right|^2
\end{align*}
and, since $V$ is continuous,
\[
\mathscr J_\varepsilon(s)
:=\mathcal B_\varepsilon(W_s,V_{s-})
-\mathcal B_\varepsilon(W_{s-},V_{s-})
-\partial_w\mathcal B_\varepsilon(W_{s-},V_{s-})\Delta W_s.
\]
For $w,v,h\in\R$, write
\[
\mathfrak D_\varepsilon(w,v;h)
:=\mathcal B_\varepsilon(w+h,v)-\mathcal B_\varepsilon(w,v)
-\partial_w\mathcal B_\varepsilon(w,v)h.
\]
Recall that the L\'evy kernel is
\[
\mathcal K_\kappa(x,dz)
=\sum_{\alpha\in\RR_+}r_\alpha(x)\delta_{\sigma_\alpha x-x}(dz),
\qquad
r_\alpha(x):=\frac{\kappa(\alpha)}{\langle\alpha,x\rangle^2}
\ind_{\{\langle\alpha,x\rangle\neq0\}},
\]
with the same conventions as before.
At a reflection jump of type $\alpha$, set
\begin{equation}\label{alpha jump}
h_{\alpha,s}
:=U_f(\sigma_\alpha X_{s-},Y_{s-})-U_f(X_{s-},Y_{s-})
=-\frac{\langle\alpha,X_{s-}\rangle}{\sqrt{\kappa(\alpha)}}J^\alpha U_f(Z_{s-}).
\end{equation}
The compensation formula \eqref{compensation formula} therefore gives
\begin{align*}
\mathbb E^{(x,\lambda)}\left[\sum_{0<s\leq T}\mathscr J_\varepsilon(s)\right]
=\mathbb E^{(x,\lambda)}\left[
\int_0^T\sum_{\alpha\in\RR_+}r_\alpha(X_{s-})
\mathfrak D_\varepsilon\bigl(W_{s-},V_{s-};h_{\alpha,s}\bigr)\,ds
\right].
\end{align*}
Thus it remains to prove the predictable estimate, almost everywhere in time,
\begin{equation}\label{main estimate}
\mathscr C_\varepsilon(s)
+\sum_{\alpha\in\RR_+}r_\alpha(X_{s-})
\mathfrak D_\varepsilon\bigl(W_{s-},V_{s-};h_{\alpha,s}\bigr)
\leq0.
\end{equation}
The two deterministic Bellman inequalities needed for this estimate are
proved next.

\subsection{The scalar Bellman inequalities}
\label{subsec:burkholder-scalar-lemmas}
Away from the coordinate axes, write
\[
\operatorname{Hess}\mathcal B(w,v)
=\begin{bmatrix}
\mathsf A(w,v)&\mathsf C(w,v)\\
\mathsf C(w,v)&\mathsf D(w,v)
\end{bmatrix},
\]
and
\[
\Delta(w,v)
:=\det\bigl(\operatorname{Hess}\mathcal B(w,v)\bigr)
=\mathsf A(w,v)\mathsf D(w,v)-\mathsf C(w,v)^2.
\]
If $K_p:=C_pp(p-1)$, direct differentiation gives
\[
\mathsf A(w,v)
=-K_p(|w|+|v|)^{p-3}\bigl((p-1)|w|+|v|\bigr),
\]
\[
\Delta(w,v)=-K_p^2(|w|+|v|)^{2p-4}.
\]
Define the Schur complement with respect to $\mathsf A$ as
\begin{equation}\label{schur}
\Lambda_p(w,v)
:=\frac{\Delta(w,v)}{\mathsf A(w,v)}
=K_p\frac{(|w|+|v|)^{p-1}}{(p-1)|w|+|v|}.
\end{equation}

At the origin, the coefficient is understood by continuous extension:
\[
\Lambda_2(0,0)=2,
\qquad
\Lambda_p(0,0)=0\quad\text{for }p>2.
\]
In particular, $\Lambda_p\geq0$. Moreover,
\begin{equation}
-\frac{\Lambda_p(w,v)}{\mathsf A(w,v)}
=\left(\frac{|w|+|v|}{(p-1)|w|+|v|}\right)^2
\leq1,
\end{equation}
and therefore $0\leq\Lambda_p(w,v)\leq-\mathsf A(w,v)$.

\begin{lemma}
\label[lemma]{lem:continuous-defect}
Fix a point $(w,v)$ at which $\operatorname{Hess}\mathcal B$ exists and let $c_0$, $\{c_\alpha\}_\alpha$ and $\{q_\alpha\}_\alpha$ be real numbers satisfying
\[
c_0^2+\sum_\alpha c_\alpha^2=1.
\]
Set
\[
d:=\sum_\alpha c_\alpha q_\alpha,
\qquad b:=c_0x+d.
\]
Then, for every $x,y\in\R$,
\[
\frac12\mathsf A(w,v)(x^2+y^2)
+\mathsf C(w,v)yb
+\frac12\mathsf D(w,v)b^2
\leq\frac12\Lambda_p(w,v)\sum_\alpha q_\alpha^2.
\]
\end{lemma}

\begin{proof}
Denote the quadratic form on the left-hand side by $F(x,y)$. Completing the square in the variable $y$ gives
\begin{align*}
F(x,y)
&=\frac12\mathsf A(w,v)x^2
+\frac12\mathsf A(w,v)\left(y+\frac{\mathsf C(w,v)}{\mathsf A(w,v)}b\right)^2
+\frac12\left(\mathsf D(w,v)-\frac{\mathsf C(w,v)^2}{\mathsf A(w,v)}\right)b^2\\
&=\frac12\mathsf A(w,v)x^2
+\frac12\mathsf A(w,v)\left(y+\frac{\mathsf C(w,v)}{\mathsf A(w,v)}b\right)^2
+\frac12\Lambda_p(w,v)b^2.
\end{align*}
Using $\mathsf A(w,v)<0$, the second term can be discarded. Since $\mathsf A(w,v)\leq-\Lambda_p(w,v)$,
\[
\frac12\mathsf A(w,v)x^2+\frac12\Lambda_p(w,v)b^2
\leq-\frac12\Lambda_p(w,v)x^2+\frac12\Lambda_p(w,v)b^2
=\frac12\Lambda_p(w,v)(b^2-x^2).
\]
By Cauchy--Schwarz,
\[
b^2
=\left(c_0x+\sum_\alpha c_\alpha q_\alpha\right)^2
\leq\left(c_0^2+\sum_\alpha c_\alpha^2\right)
\left(x^2+\sum_\alpha q_\alpha^2\right)
=x^2+\sum_\alpha q_\alpha^2,
\]
hence $b^2-x^2\leq\sum_\alpha q_\alpha^2$, which gives the desired result.
\end{proof}

\begin{lemma}
\label[lemma]{lem:finite-difference}
For every $(w,v)\in\R^2$ and every $h\in\R$,
\[
\mathcal B(w+h,v)-\mathcal B(w,v)-\partial_w\mathcal B(w,v)h
\leq-\frac12\Lambda_p(w,v)h^2,
\]
where $\partial_w\mathcal B(0,v):=0$.
\end{lemma}

\begin{proof}
For $p=2$, we have $\mathcal B(w,v)=v^2-w^2$ and $\Lambda_2(w,v)=2$, so the claim holds with equality. We may therefore assume that $p>2$. Fix $y:=|v|$ and define, for $t\geq0$,
\[
\Phi(t):=C_p\bigl(y-(p-1)t\bigr)(t+y)^{p-1}.
\]
Then $\mathcal B(w,v)=\Phi(|w|)$. By evenness in the first variable, it is enough to consider a base point $w\geq0$. Assume first that $w>0$ and put $\Lambda:=\Lambda_p(w,y)>0$. For $t\geq0$, define
\[
F(t):=\Phi(t)-\Phi(w)-\Phi'(w)(t-w)+\frac12\Lambda(t-w)^2.
\]
We have $F(w)=F'(w)=0$, and we first show that $F(t)\leq0$ for $t\geq0$. Then, evaluating at $t=w+h\geq0$ gives the desired inequality. The case $w+h<0$ will be treated later. Define $G(t):=-\Phi''(t)$. Direct computation gives
\begin{gather*}
G(t)=K_p(t+y)^{p-3}\bigl((p-1)t+y\bigr),
\qquad
G'(t)=K_p(p-2)(t+y)^{p-4}\bigl((p-1)t+2y\bigr)\geq0.
\end{gather*}
Hence $G$ is nondecreasing. At the base point,
\[
\frac{G(w)}\Lambda
=\left(\frac{(p-1)w+y}{w+y}\right)^2\geq1.
\]
Consequently, if $t\geq w$, then
\[
F''(t)=\Lambda-G(t)\leq\Lambda-G(w)\leq0.
\]
Since $F'(w)=0$, it follows that $F'(t)\leq0$ and therefore $F(t)\leq F(w)=0$ for every $t\geq w$.

It remains to consider $0\leq t\leq w$. Since $G$ is nondecreasing, $F''$ is nonincreasing and can change sign at most once, from nonnegative to nonpositive. Thus $F'$ is first nondecreasing and then nonincreasing. Moreover,
\[
F'(0)=-\bigl(\Phi'(w)+\Lambda w\bigr)
=K_p(p-2)\frac{w^2(w+y)^{p-2}}{(p-1)w+y}\geq0,
\]
whereas $F'(w)=0$. It follows that $F'(t)\geq0$ on $[0,w]$, and hence
\[
F(t)\leq F(w)=0,\qquad0\leq t\leq w.
\]

We have proved the estimate when $w+h\geq0$. To handle the case $w+h<0$, define
\[
\widetilde F(s)
:=\Phi(|s|)-\Phi(w)-\Phi'(w)(s-w)+\frac12\Lambda(s-w)^2,
\qquad s\in\R.
\]
For $t\geq0$, the preceding argument gives $\widetilde F(t)=F(t)\leq0$, while
\[
\widetilde F(-t)
=F(t)+2t\bigl(\Phi'(w)+\Lambda w\bigr)
\leq F(t)\leq0.
\]
Therefore $\widetilde F(s)\leq0$ for every $s\in\R$, and taking $s=w+h$ proves the result when $w>0$. \\
Finally, let $w=0$ and put $\Lambda:=\Lambda_p(0,y)$. Since $G$ is nondecreasing and equals $\Lambda$ at the origin, the function
\[
t\longmapsto\Phi(t)+\frac12\Lambda t^2
\]
is concave on $[0,\infty)$. Its derivative at the origin is $\Phi'(0)=0$, and therefore
\[
\Phi(t)-\Phi(0)\leq-\frac12\Lambda t^2,
\qquad t\geq0.
\]
The same conclusion holds for negative $t$ by evenness. This completes the proof.
\end{proof}

\subsubsection{Stability of the Bellman inequalities under regularization}

For $\varepsilon>0$, define
\[
\Lambda_{p,\varepsilon}:=\Lambda_p*\rho_\varepsilon.
\]
Similarly, we denote the entries of the regularized Hessian by
\[
\mathsf A_\varepsilon:=\partial_{ww}\mathcal B_\varepsilon,
\qquad
\mathsf C_\varepsilon:=\partial_{wv}\mathcal B_\varepsilon,
\qquad
\mathsf D_\varepsilon:=\partial_{vv}\mathcal B_\varepsilon.
\]
In the following lemma, we prove that the two lemmas above remain true for the regularized Bellman function $\mathcal B_\varepsilon$.
\begin{lemma}[Stability under regularization]
\label[lemma]{lem:bellman-regularization}
For every $\varepsilon>0$, the estimates of \cref{lem:continuous-defect,lem:finite-difference}
remain valid for the regularized Bellman function
$\mathcal B_\varepsilon=\mathcal B*\rho_\varepsilon$.
More precisely, if
\[
b=c_0x+\sum_\alpha c_\alpha q_\alpha,
\qquad c_0^2+\sum_\alpha c_\alpha^2=1,
\]
then
\[
\frac12\mathsf A_\varepsilon(w,v)(x^2+y^2)
+\mathsf C_\varepsilon(w,v)yb
+\frac12\mathsf D_\varepsilon(w,v)b^2
\leq\frac12\Lambda_{p,\varepsilon}(w,v)\sum_\alpha q_\alpha^2.
\]
Moreover, for every $h\in\R$,
\[
\mathcal B_\varepsilon(w+h,v)-\mathcal B_\varepsilon(w,v)
-\partial_w\mathcal B_\varepsilon(w,v)h
\leq-\frac12\Lambda_{p,\varepsilon}(w,v)h^2.
\]
\end{lemma}

\begin{proof}
The only point requiring some care is that $\mathcal B$ is not globally
$C^2$. We therefore interpret its second derivatives distributionally.
The first derivative in the $w$-variable is continuous across $w=0$,
while across $v=0$ one has
\[
D_{vv}\mathcal B
=\partial_{vv}\mathcal B\,\mathcal L^2
-2C_pp(p-2)|w|^{p-1}\,
\mathcal L^1(dw)\otimes\delta_0(dv).
\]
For fixed $x,y,b$ and $\{q_\alpha\}_\alpha$, the distribution
\[
\mathscr Q
=\frac12D_{ww}\mathcal B\,(x^2+y^2)
+D_{wv}\mathcal B\,yb
+\frac12D_{vv}\mathcal B\,b^2
-\frac12\Lambda_p\sum_\alpha q_\alpha^2
\]
is nonpositive. Indeed, its absolutely continuous part is nonpositive
by \cref{lem:continuous-defect}, while its singular part is
\[
-C_pp(p-2)b^2|w|^{p-1}\,
\mathcal L^1(dw)\otimes\delta_0(dv),
\]
which is nonpositive for $p\geq2$. Convolving $\mathscr Q$ with the nonnegative mollifier
$\rho_\varepsilon$ and using that distributional derivatives
commute with convolution gives immediately
\[
\frac12\mathsf A_\varepsilon(x^2+y^2)
+\mathsf C_\varepsilon yb
+\frac12\mathsf D_\varepsilon b^2
\leq\frac12\Lambda_{p,\varepsilon}\sum_\alpha q_\alpha^2.
\]

The finite-difference estimate is even more direct. By convolution, the left-hand side of the desired inequality is equal to
\[
\int_{\R^2}\Bigl[
\mathcal B(w-\eta+h,v-\zeta)
-\mathcal B(w-\eta,v-\zeta)
-\partial_w\mathcal B(w-\eta,v-\zeta)h
\Bigr]\rho_\varepsilon(\eta,\zeta)\,d\eta\,d\zeta,
\]
and applying \cref{lem:finite-difference} gives the desired bound.
\end{proof}

\begin{remark}
Since $\mathcal B_\varepsilon$ is smooth, dividing the last inequality in the preceding proof by $h^2$ and letting $h\to0$ also yields
\[
\mathsf A_\varepsilon(w,v)\leq-\Lambda_{p,\varepsilon}(w,v).
\]
\end{remark}

\subsection{Completion of the scalar proof}

We now prove \eqref{main estimate}. At a fixed time $s<T$, abbreviate
\[
x:=\partial_jU_f(Z_{s-}),
\qquad y:=\partial_yU_f(Z_{s-}),
\]
\[
q_\alpha:=J^\alpha U_f(Z_{s-}),
\qquad b:=\frac{T_jU_f(Z_{s-})}{\theta_j}.
\]
Discarding the terms containing $\partial_kU_f$, $k\neq j$, which have nonpositive coefficients, and applying \cref{lem:bellman-regularization}, we obtain
\[
\mathscr C_\varepsilon(s)
\leq\frac12\Lambda_{p,\varepsilon}(W_{s-},V_{s-})
\sum_{\alpha\in\RR_+}|J^\alpha U_f(Z_{s-})|^2.
\]

For $h_{\alpha,s}$ as in \eqref{alpha jump}, the regularized finite-difference estimate gives
\[
\mathfrak D_\varepsilon\bigl(W_{s-},V_{s-};h_{\alpha,s}\bigr)
\leq-\frac12\Lambda_{p,\varepsilon}(W_{s-},V_{s-})|h_{\alpha,s}|^2.
\]
Since, almost everywhere in time,
\[
r_\alpha(X_{s-})|h_{\alpha,s}|^2=|J^\alpha U_f(Z_{s-})|^2,
\]
the compensated jump drift satisfies
\[
\sum_{\alpha\in\RR_+}r_\alpha(X_{s-})
\mathfrak D_\varepsilon\bigl(W_{s-},V_{s-};h_{\alpha,s}\bigr)
\leq-\frac12\Lambda_{p,\varepsilon}(W_{s-},V_{s-})
\sum_{\alpha\in\RR_+}|J^\alpha U_f(Z_{s-})|^2.
\]
Thus the possible positive defect of the continuous drift is cancelled exactly
by the negative contribution of the compensated reflection jumps:
\[
\mathscr C_\varepsilon(s)
+\sum_{\alpha\in\RR_+}r_\alpha(X_{s-})
\mathfrak D_\varepsilon\bigl(W_{s-},V_{s-};h_{\alpha,s}\bigr)
\leq0.
\]

Returning to the localized It\^o formula and using the compensation
identity, we obtain, for $dw$-almost every $x\in\R^n$,
\[
\mathbb E^{(x,\lambda)}\mathcal B_\varepsilon(W_T,V_T)
\leq\mathcal B_\varepsilon(P_\lambda f(x),0).
\]
For fixed $x,t$ and $N$, the stopped processes are bounded. Hence the
locally uniform convergence
$\mathcal B_\varepsilon\to\mathcal B$ allows us to let
$\varepsilon\downarrow0$ and conclude that
\[
\mathbb E^{(x,\lambda)}\mathcal B(W_T,V_T)
\leq\mathcal B(P_\lambda f(x),0)\leq0.
\]
The majorization property of $\mathcal B$ and conditional Jensen's inequality, using $W_T=\mathbb E^{(x,\lambda)}[f(X_\tau)\mid\Sigma_T]$, now yield
\[
\mathbb E^{(x,\lambda)}|V_T|^p
\leq(p-1)^p\mathbb E^{(x,\lambda)}|f(X_\tau)|^p.
\]
Integrating in $x$, using \cref{lambda exp}, and then applying Fatou's lemma twice, we obtain
\begin{equation*}
\mathbb E^\lambda|V_\tau|^p
\leq(p-1)^p\|f\|_{L^p(dw)}^p.
\end{equation*}
The conditional projection and the passage $\lambda\to\infty$ are
identical to those in the preceding section, and
$\mathcal T_{L_j}=\frac12R_j$ proves the stated estimate for $p\geq2$. For $1<p<2$, the identity $R_j^*=-R_j$, duality and the estimate for $p'$ give
\[
\|R_j\|_{L^p(dw)\to L^p(dw)}
=\|R_j\|_{L^{p'}(dw)\to L^{p'}(dw)}
\leq2(p'-1)\theta_j
=2(p^*-1)\theta_j.
\]
\end{proof}

\subsection{The non-invariant Riesz vector}

\begin{proof}[Proof of \cref{thm:Riesz-vector-noninvariant}]
Let $p\geq2$ and initially take $f\in\mathcal S(\R^n)$. Fix $\lambda>0$ and use $\tau$, $Z$ and the uncentered source martingale $W_t=U_f(Z_{t\wedge\tau})$ as in the scalar proof. Define the
$\R^n$-valued continuous martingale
\[
\mathbf V_t
:=\frac1{\Theta_\kappa}\int_0^{t\wedge\tau}\nabla_\kappa U_f(Z_{s-})\,dY_s.
\]
Its $j$th component is
\[
\mathbf V_t^j=\frac1{\Theta_\kappa}(L_j\ast f)_{t\wedge\tau}.
\]
The argument follows the scalar proof, with one additional feature: the Hessian in the vector variable has both a radial and a tangential component. We first isolate the deterministic estimate needed to control these two contributions. Write
\[
a:=\nabla_xU_f,
\qquad
q:=(J^\alpha U_f)_{\alpha\in\RR_+},
\qquad
b:=\frac1{\Theta_\kappa}\nabla_\kappa U_f
=\frac1{\Theta_\kappa}\mathsf S_\kappa(a,q).
\]
By the definition of $\Theta_\kappa$, we have $|b|^2\leq|a|^2+|q|^2$.
Consider, for $(w,v)\in\R\times\R^n$, the vector-valued analogue of the scalar Bellman function
\[
\mathcal B(w,v):=\mathcal B(w,|v|)
=C_p\bigl(|v|-(p-1)|w|\bigr)\bigl(|w|+|v|\bigr)^{p-1}.
\]
For $v\neq0$, set $r:=|v|$ and $e:=v/|v|$. Decomposing any vector $b\in\R^n$ into radial and tangential components,
\[
b_0:=\langle b,e\rangle,
\qquad b_1:=b-b_0e,
\]
the corresponding quadratic form simplifies to
\[
\left\langle\operatorname{Hess}_v\mathcal B(w,v)b,b\right\rangle
=\mathsf D(w,|v|)b_0^2+\mathsf E(w,|v|)|b_1|^2,
\]
where $\mathsf E(w,r):=\partial_r\mathcal B(w,r)/r$.
As in the scalar setting, the Schur complement $\Lambda_p(w,r)$ is given by \eqref{schur}, and
\begin{align*}
\mathsf E(w,r)
&=C_pp(|w|+r)^{p-2}\frac{r-(p-2)|w|}{r},\\
\Lambda_p(w,r)-\mathsf E(w,r)
&=C_pp(p-2)(|w|+r)^{p-2}
\frac{(p-1)|w|^2+|w|r+r^2}{r\bigl((p-1)|w|+r\bigr)}\geq0.
\end{align*}
Together with $\Lambda_p\leq-\mathsf A$, this gives
\begin{equation}\label{eq:vector-radial-coefficients}
\mathsf E(w,r)\leq\Lambda_p(w,r)\leq-\mathsf A(w,r).
\end{equation}
The following estimate is the vector counterpart of \cref{lem:continuous-defect}. It is the only additional deterministic ingredient needed for the Riesz vector.

\begin{lemma}
\label[lemma]{lem:vector-continuous-defect}
Let $a\in\R^n$, $q\in\R^m$, $r\in\R$, and set
\[
b:=\frac1{\Theta_\kappa}\mathsf S_\kappa(a,q).
\]
Let $(w,v)$ satisfy $w\neq0$ and $v\neq0$, and let $e,b_0,b_1$ be as above. Then
\begin{align*}
\frac12\mathsf A(w,|v|)\bigl(|a|^2+r^2\bigr)
+\mathsf C(w,|v|)rb_0
+\frac12\mathsf D(w,|v|)b_0^2
+\frac12\mathsf E(w,|v|)|b_1|^2
\leq\frac12\Lambda_p(w,|v|)|q|^2.
\end{align*}
\end{lemma}

\begin{proof}
Completing the square in $r$ isolates the Schur complement $\Lambda_p(w,|v|)$:
\begin{align*}
\frac12\mathsf A(w,|v|)r^2
+\mathsf C(w,|v|)rb_0
+\frac12\mathsf D(w,|v|)b_0^2
&=\frac12\mathsf A(w,|v|)
\left(r+\frac{\mathsf C(w,|v|)}{\mathsf A(w,|v|)}b_0\right)^2
+\frac12\Lambda_p(w,|v|)b_0^2\\
&\leq\frac12\Lambda_p(w,|v|)b_0^2,
\end{align*}
where we discarded the first term since $\mathsf A(w,|v|)<0$. Using \eqref{eq:vector-radial-coefficients}, we get
\begin{align*}
\frac12\mathsf A(w,|v|)|a|^2
+\frac12\Lambda_p(w,|v|)b_0^2
+\frac12\mathsf E(w,|v|)|b_1|^2
&\leq\frac12\Lambda_p(w,|v|)\bigl(b_0^2+|b_1|^2-|a|^2\bigr)\\
&=\frac12\Lambda_p(w,|v|)\bigl(|b|^2-|a|^2\bigr).
\end{align*}
Since $|b|^2-|a|^2\leq|q|^2$ by the definition of $\Theta_\kappa$, the claim follows.
\end{proof}

The passage to a smooth Bellman function is the same as in the scalar case.
Choose a nonnegative mollifier $\rho\in C_c^\infty(\R\times\R^n)$ which is even in the first variable, radial in the second variable, and has integral one. Put $\rho_\varepsilon(z):=\varepsilon^{-(n+1)}\rho(z/\varepsilon)$ and set
\[
\mathcal B_\varepsilon:=\mathcal B*\rho_\varepsilon,
\qquad
\Lambda_{p,\varepsilon}:=\Lambda_p*\rho_\varepsilon,
\]
where $\Lambda_p$ is interpreted radially in the second variable.
For $n\geq2$, the $1/|v|$ singularity of the tangential Hessian coefficient is locally integrable and creates no additional measure supported on $v=0$. For $n=1$, the singular contribution is the nonpositive measure treated in \cref{lem:bellman-regularization}. Thus the same distributional argument gives
\begin{align*}
\frac12\partial_{ww}\mathcal B_\varepsilon(w,v)\bigl(|a|^2+r^2\bigr)
+rD^2_{wv}\mathcal B_\varepsilon(w,v)[b]
+\frac12D^2_{vv}\mathcal B_\varepsilon(w,v)[b,b]
\leq\frac12\Lambda_{p,\varepsilon}(w,v)|q|^2.
\end{align*}
The finite-difference estimate requires no new argument. Indeed, for each fixed $v\in\R^n$, the function $w\longmapsto\mathcal B(w,v)$ is the scalar Bellman function with second variable $|v|$. Hence the regularized analogue of \cref{lem:finite-difference} gives
\[
\mathcal B_\varepsilon(w+h,v)-\mathcal B_\varepsilon(w,v)
-\partial_w\mathcal B_\varepsilon(w,v)h
\leq-\frac12\Lambda_{p,\varepsilon}(w,v)h^2.
\]
We now apply It\^o's formula to $\mathcal B_\varepsilon(W,\mathbf V)$ under each initial law $\mathbb P_{(x,\lambda)}$, with the same localization as in the scalar proof. Its continuous drift density is
\begin{align*}
\mathscr C_\varepsilon(s)
={}&\frac12\partial_{ww}\mathcal B_\varepsilon(W_{s-},\mathbf V_{s-})
\left(|\nabla_xU_f(Z_{s-})|^2+|\partial_yU_f(Z_{s-})|^2\right)\\
&+\frac{\partial_yU_f(Z_{s-})}{\Theta_\kappa}
\left\langle\nabla_v\partial_w\mathcal B_\varepsilon(W_{s-},\mathbf V_{s-}),
\nabla_\kappa U_f(Z_{s-})\right\rangle\\
&+\frac1{2\Theta_\kappa^2}
\left\langle\operatorname{Hess}_v\mathcal B_\varepsilon(W_{s-},\mathbf V_{s-})
\nabla_\kappa U_f(Z_{s-}),\nabla_\kappa U_f(Z_{s-})\right\rangle.
\end{align*}
Applying the regularized vector estimate with $a,q,b$ as above and $r=\partial_yU_f(Z_{s-})$ gives
\[
\mathscr C_\varepsilon(s)
\leq\frac12\Lambda_{p,\varepsilon}(W_{s-},\mathbf V_{s-})
\sum_{\alpha\in\RR_+}|J^\alpha U_f(Z_{s-})|^2.
\]
Analogously, the compensated jump drift is
\begin{align*}
\mathscr J_\varepsilon(s)
:=\sum_{\alpha\in\RR_+}r_\alpha(X_{s-})
\Bigl[&\mathcal B_\varepsilon(W_{s-}+h_{\alpha,s},\mathbf V_{s-})
-\mathcal B_\varepsilon(W_{s-},\mathbf V_{s-})\\
&-\partial_w\mathcal B_\varepsilon(W_{s-},\mathbf V_{s-})h_{\alpha,s}\Bigr],
\end{align*}
where $h_{\alpha,s}$ is defined in \eqref{alpha jump}. The regularized finite-difference estimate gives
\[
\mathscr J_\varepsilon(s)
\leq-\frac12\Lambda_{p,\varepsilon}(W_{s-},\mathbf V_{s-})
\sum_{\alpha\in\RR_+}|J^\alpha U_f(Z_{s-})|^2.
\]
Hence $\mathscr C_\varepsilon(s)+\mathscr J_\varepsilon(s)\leq0$. The same arguments as in the scalar proof give
\[
\|\mathbf Rf\|_{L^p(dw;\ell_n^2)}
\leq2(p-1)\Theta_\kappa\|f\|_{L^p(dw)}.
\]
\end{proof}

\begin{remark}
The proof uses $p\geq2$ both in the finite-difference estimate and in the comparisons
\[
\mathsf E\leq\Lambda_p\leq-\mathsf A,
\]
which control the tangential component of the vector Hessian. The scalar bound below $2$ follows from $R_j^*=-R_j$. The same duality argument does not give the corresponding vector bound, since the adjoint of $\mathbf R$ is the Riesz divergence
$\mathbf g\mapsto-\sum_{j=1}^nR_jg_j$, rather than a scalar-to-vector operator. Interpolating with the $L^2$ estimate gives
\end{remark}

\begin{corollary}
\label[corollary]{cor:interpolated-riesz-bounds}
Let $q>2$ and $2\leq p\leq q$, and set
\[
\vartheta(p,q):=\frac{q(p-2)}{p(q-2)}.
\]
Then
\[
\|\mathbf Rf\|_{L^p(dw;\ell_n^2)}
\leq\bigl[2(q-1)\Theta_\kappa\bigr]^{\vartheta(p,q)}
\|f\|_{L^p(dw)}.
\]
In dimension one, for every $1<p<\infty$ and every $q>2$ with $q\geq p^*$,
\[
\|H_\kappa f\|_{L^p(dw)}
\leq\bigl[4(q-1)\bigr]^{\vartheta(p^*,q)}\|f\|_{L^p(dw)},
\]
uniformly in $\kappa$. Both estimates give the exact constant $1$ at $p=2$.
\end{corollary}

\section*{AI disclosure}
The authors used OpenAI's GPT-5.6 Sol as an assistive tool for refinement of the exposition and checks of algebraic consistency. The underlying proof strategies were developed by the authors, who take full responsibility for the content of the manuscript.

\appendix

\bibliographystyle{amsplain}
\bibliography{biblio}
\end{document}